\documentclass[12pt,letterpaper]{article}
\usepackage[top=2.2cm, bottom=2.2cm, left=2cm, right=2cm]{geometry}
\usepackage{amsmath,amsfonts,amssymb}
\usepackage{tikz-cd}
\usepackage{appendix}
\usepackage{amsthm}

\DeclareMathOperator{\Spec}{\operatorname{Spec}}
\newcommand{\Z}{\mathbb{Z}}
\newcommand{\Q}{\mathbb{Q}}
\newcommand{\CC}{\mathbb{C}}
\DeclareMathOperator{\Ho}{\mathop{\mathrm{Ho}}}
\DeclareMathOperator{\hocolim}{\mathop{\mathrm{hocolim}}}
\DeclareMathOperator{\colim}{\mathop{\mathrm{colim}}}
\DeclareMathOperator{\const}{\mathop{\mathrm{const}}}
\DeclareMathOperator{\chara}{\mathop{\mathrm{char}}}

\theoremstyle{plain}
\newtheorem{theorem}{Theorem}[section]
\newtheorem{varthm}{Theorem}

\newtheorem{lemma}{Lemma}[section]
\newtheorem{proposition}{Proposition}[section]
\newtheorem{corollary}{Corollary}[section]

\theoremstyle{definition}
\newtheorem{definition}{Definition}[section]

\theoremstyle{remark}
\newtheorem{remark}{Remark}[section]
\newtheorem{example}{Example}[section]

\usepackage{enumitem}

\title{Affine models for Noetherian schemes}
\author{Alexey G.~Gorinov, Egor S.~Kosolapov}
\date{}

\usepackage{hyperref}
\begin{document}
\maketitle

\begin{abstract}
Let $S$ be a base scheme, assumed separated and Noetherian. We define \emph{adequate classes} of morphisms of $S$-schemes by formalizing certain properties of homotopy equivalences of complex algebraic varieties. Other examples of adequate classes include morphisms of varieties over an algebraically closed field of positive characteristic which induce an isomorphism of \'etale cohomology, and weak equivalences in the cdh topology. An \emph{affine model} for a Noetherian scheme $X$ over $S$ is an affine scheme $M_X$ equipped with an adequate morphism $M_X\to X$. 

In this paper we construct affine models for arbitrary separated schemes of finite type over $S$. Our construction can be viewed as a generalization of the Jouanolou trick. As an application, we construct a mixed Hodge structure on the Leray spectral sequence of an arbitrary proper morphism $f:X\to Y$ of complex algebraic varieties, generalizing an argument by Donu Arapura which assumed $Y$ quasi-projective and $f$ projective.
\end{abstract}

\section{Introduction}
Let $S$ be a base scheme, which we assume separated and Noetherian unless stated otherwise.

\begin{definition}
A class $\mathcal{W}$ of morphisms in a subcategory $\mathcal{C}$ of schemes over $S$ is called \emph{adequate} if the following conditions are satisfied:
\begin{itemize}
\item $\mathcal{W}$ is closed under composition;
\item $\mathcal{W}$ contains the projections $X \times_S \mathbb{A}^1_S \to X$;
\item For any morphism $f: X \to Y$ in $\mathcal{C}$ and any cover of $X = X_1 \cup X_2$ by closed subschemes and $Y = Y_1 \cup Y_2$ by open subschemes, if $f(X_i) \subset Y_i$ and the induced morphisms $X_i \to Y_i$ and $X_1 \cap X_2 \to Y_1 \cap Y_2$ belong to $\mathcal{W}$, then $f \in \mathcal{W}$.
\end{itemize}
\end{definition}

One can think of arrows in $\mathcal{W}$ as homotopy equivalences of some type or another, see examples below.

An \emph{affine model} for a scheme $X$ is a scheme $Y \in \mathcal{C}$ affine over $\Z$ with a morphism $m: Y \to X$ that belongs to $\mathcal{W}$. One of the main results of this paper is as follows:

\begin{varthm}[Theorem~\ref{chap2:main_thrm}]\label{thm:main_intro}
    Let $\mathcal{C}$ be the category of separated schemes of finite type over a separated Noetherian scheme $S$, and let $\mathcal{W}$ be any class of adequate morphisms of $\mathcal{C}$. Then every scheme $X \in \mathcal{C}$ admits an affine model.
\end{varthm}

\begin{remark}
The affine models from the theorem will also be affine over $S$ by Lemma~\ref{chap2:lemma_closed_affine_embeddings}.
\end{remark}

We will now give a few examples of classes of adequate morphisms which we will explore in more detail later. 
In all these examples $\mathcal{C}$ will be as in Theorem~\ref{thm:main_intro}.

\begin{example}\label{ex:intro_spec_c}
    Let $S = \mathop{\mathrm{Spec}}(\mathbb{C})$. A morphism $f: X \to Y$  in $\mathcal{C}$ belongs to $\mathcal{W}$ if and only if it induces a homotopy equivalence $X(\CC)\to Y(\CC)$ of the spaces of complex points in the complex analytic topology.
\end{example}

\begin{example}\label{ex:intro_char_p}
        Let $S = \mathop{\mathrm{Spec}}(k)$ where $k$ is a field of characteristic $p>0$. A morphism $f: X \to Y$ in $\mathcal{C}$ belongs to $\mathcal{W}$ if and only if it induces an isomorphism
        \[
        H^{*}_{\mbox{{\scriptsize \'et}}}(Y; \mathcal{F}) \to H^{*}_{\mbox{{\scriptsize \'et}}}(X; f^{*}\mathcal{F})
        \]
        for any $l$-torsion sheaf $\mathcal{F}$ on $Y$ with $l\neq p$ a prime.
\end{example}

\begin{example}\label{ex:intro_cdh}
    Let $S$ be an arbitrary separated Noetherian base scheme. Then we can take $\mathcal{W}$ to be the class of morphisms that become weak equivalences in the model structure on simplicial presheaves induced by the cdh topology.
\end{example}

We discuss Examples~\ref{ex:intro_spec_c}, \ref{ex:intro_char_p} and~\ref{ex:intro_cdh} in detail in Sections~\ref{example_spec_c}, \ref{example_etale_case} and~\ref{example_cdh} respectively.

\begin{remark}
We do not know if one can replace cdh by Nisnevich in Example~\ref{ex:intro_cdh}.
\end{remark}

Affine models were inspired by the \emph{Jouanolou trick}~\cite{Jo73}, which we will now briefly review. Let $S=\Spec (k)$ where $k$ is an algebraically closed field. Suppose $X$ is a quasi-projective variety over~$k$. We want to construct an affine algebraic variety $M_X$ over~$k$, and a morphism $M_X\to X$ of $k$-varieties which would be a Zariski locally trivial fibre bundle with fibre $\mathbb{A}^n$ (for some $n$). We start with $X=\mathbb{P}^m$ and take $M_{\mathbb{P}^m}$ to be the subvariety of the affine space of all $(m+1)\times(m+1)$-matrices $A$ over $k$ given by the equations $A^2=A$ and $\mathop{\mathrm{rk}} A=1$. On the level of $k$-points, the morphism $M_{\mathbb{P}^m}\to \mathbb{P}^m$ takes a matrix $A$ to its image $\in\mathbb{P}^m(k)$.

If $X\subset\mathbb{P}^m$ is closed, we set $M_X=X\times_{\mathbb{P}^m}M_{\mathbb{P}^m}$. If $X$ is arbitrary quasi-projective, let $\bar X\supset X$ be the closure of $X$ in some $\mathbb{P}^m\supset X$. We may assume that after blowing up $\bar X-X$ if necessary the complement $\bar X-X$ is an effective Cartier divisor. Then the inclusion $X\to \bar X$ is affine. Set $M_X=X\times_{\mathbb{P}^m}M_{\mathbb{P}^m}$ and $\bar M_X=\bar X\times_{\mathbb{P}^m}M_{\mathbb{P}^m}$. Since the $X\to \bar X$ is affine, so is the inclusion $M_X\to\bar M_X$. As $\bar M_X$ is affine, we conclude that so is $M_X$. 

\smallskip

This construction was later generalized by Thomason~\cite{We89} to schemes having ample families of line bundles, i.e.\ to divisorial schemes. 
Let us mention a few examples, as well as a non-example:

\begin{example}
    Any projective scheme $X$ over an affine base $S$ is divisorial. 
    A quasi-projective variety over an algebraically closed field is divisorial.
\end{example}

\begin{example}
    Every separated, locally factorial Noetherian scheme $X$ over a field is divisorial. This follows from the results of~\cite{Bo91}. In particular, every smooth scheme over a field is divisorial.
\end{example}

\begin{example}
    A proper variety $X$ over an algebraically closed field $k$ with $\operatorname{Pic}(X)= 0$ cannot be divisorial. 
    To give a concrete example, in~\cite{Ei92}, Example 3.5, Eikelberg constructs a proper toric variety $X_\Sigma$ with $\operatorname{Pic}(X_\Sigma) = 0$ defined by a fan $\Sigma$ derived from the faces of a triangular prism with one vertex perturbed. 
\end{example}

In view of these examples it is natural to look for an analog of Jouanolou's and Thomason's constructions which works for a larger class of schemes but retains some of the useful properties. Among these properties, the first one that comes to mind is that over $\CC$, starting from a variety $X$ we get an affine variety $M_X$ and a morphism $M_X\to X$ which is a homotopy equivalence on the complex points. Axiomatizing this idea one arrives at the notion of an affine model. 

We will now briefly outline our construction and mention a few applications.

Assume for simplicity that the base $S=\Spec (k)$ where $k$ is an algebraically closed field, and $X = U_1 \cup U_2$ where $U_1$ and $U_2$ are affine open subschemes. Consider the pushout square

\[
\begin{tikzcd}
    U_1 \cap U_2 \arrow[hook]{r}{\iota_2} \arrow[hook]{d}{\iota_1} & U_2 \arrow{d} \\
    U_1 \arrow{r} & X.
\end{tikzcd}
\]

As $X$ is separated, $U_1 \cap U_2$ is affine. Let $e: U_1 \cap U_2 \hookrightarrow \mathbb{A}^n_k$ be a closed embedding. We then consider the diagram

\[
\begin{tikzcd}
    U_1 \cap U_2 \arrow[hook]{r}{(\iota_2, e)} \arrow[hook]{d}{(\iota_1, e)} & U_2 \times_k \mathbb{A}^n_k \\
    U_1 \times_k \mathbb{A}^n_k.
\end{tikzcd}
\]

We will see that
\begin{itemize}
    \item $(\iota_1, e)$ and $(\iota_2, e)$ are closed embeddings;
    \item The pushout $M_X$ of this diagram (in schemes over $k$) exists and is affine;
    \item The induced morphism $m_X: M_X \to X$ lies in $\mathcal{W}$.
\end{itemize}

Note that the fibers of $m$ over closed points of $U_1 - U_2$ and $U_2 - U_1$ are isomorphic to $\mathbb{A}^n_k$, while over closed points of $U_1 \cap U_2$ they consist of two copies of $\mathbb{A}^n_k$ glued along a point. So typically, $m_X$ will not be an affine bundle, and $M_X$ will not or smooth or irreducible even if $X$ is, in contrast to what one gets by applying Jouanolou's or Thomason's constructions. This notwithstanding, in some respects $m_X$ does resemble a locally trivial bundle with contractible fiber:  
%
%

\begin{varthm}[Theorem~\ref{chap3:main_thm_Spec_C}]\label{thm:intro_ho_eq}
For every complex algebraic variety $X$ there exist an affine complex algebraic variety $M_X$ and a morphism $m_X:M_X\to X$ which induces a homotopy equivalence $M_X(\CC)\to X(\CC)$.

Moreover, $m_X$ has contractible fibers, induces a (hyper)cohomology isomorphism for every bounded below complex of sheaves on $X(\CC)$, and the base change theorem is true for pulling back arbitrary morphisms of complex algebraic varieties along $m_X$.
\end{varthm}

As an application, we generalize a result of Arapura, proven in~\cite{Ar03} for projective morphisms with quasi-projective target variety:

\begin{varthm}[Theorem~\ref{chap4:main_thm}]\label{thm:intro_leray_is_motivic}
    Let $f: X \to Y$ be an arbitrary proper morphism between arbitrary complex algebraic varieties $X$ and $Y$. Then the Leray spectral sequence associated with this morphism is motivic.
\end{varthm}

Informally, a spectral sequence being motivic means that any functorial construction on the cohomology groups of algebraic varieties (such as e.g.\ mixed Hodge structures) that is compatible with the long exact sequence of relative cohomology groups extends to this spectral sequence. This yields the following corollary:

\begin{varthm}[Corollary~\ref{chap4:main_corollary}]\label{cor:intro_leray_has_mhs}
    For an arbitrary proper morphism $f: X \to Y$ of complex algebraic varieties, the Leray spectral sequence associated with $f$ carries a natural mixed Hodge structure.
\end{varthm}

The idea of the proof of Theorem~3.1 in~\cite{Ar03} is as follows: Let $f:X\to Y$ be a morphism of complex algerbaic varieties. Firstly, one can define a spectral sequence associated with an exhaustive filtration by closed subsets for any topological space. In the algebraic case these spectral sequences are clearly motivic.

Secondly, when the target $Y$ of the morphism $f: X \to Y$ is affine, there exists a filtration of $Y$ by closed subvarieties such that the spectral sequence obtained by filtering $X$ by the preimages is isomorphic to the Leray spectral sequence starting from the second page. This parallels the construction of the Leray spectral sequence for a fiber bundle built from a filtration of $Y$ by cellular skeleta. In fact, by Beilinson's lemma (see Lemma~\ref{chap4:Beilinson_lemma}) one can generalize the notion of a cellular filtration to any morphism $f: X \to Y$ and any constructible sheaf $\mathcal{F}$ on $X$. 

Thirdly, if the target $Y$ is not necessarily affine, one uses the Jouanolou trick to construct an affine variety $M_Y$ and a morphism $M_Y\to Y$ such that the Leray spectral sequence of $f:X\to Y$ and of the pullback $X\times_Y M_Y\to M_Y$ are isomorphic.

Our proof of Theorem~\ref{thm:intro_leray_is_motivic} follows the same plan up until the last step, where we use our construction of affine models, the Jouanolou trick not being available for non-quasi-projective varieties. For affine models constructed using the Jouanolou trick it is straightforward to check that the Leray spectral sequences of $X\to Y$ and $X\times_Y M_Y\to M_Y$ are isomorphic. In our case this is still true but requires some work.

\smallskip

Theorems~\ref{thm:intro_ho_eq} and~\ref{thm:intro_leray_is_motivic} have \'etale analogs, namely Theorems~\ref{chap3:main_thrm_etale_case} and~\ref{chap4:etale_case_motivic_nature} respectively.

\smallskip

\begin{remark}
In a companion paper~\cite{egor_non_proper} Theorems~\ref{thm:intro_leray_is_motivic} and~\ref{cor:intro_leray_has_mhs} will be extended to the case when $f$ is not necessarily proper. 
\end{remark}

\begin{remark}
Theorem~\ref{thm:intro_ho_eq} has an analog for complex analytic spaces such that the underlying topological space is Hausdorff and has a countable dense subset; affine schemes get replaced with Stein spaces. 
The details will appear elsewhere.
\end{remark}

\begin{remark}
Another approach to equipping the Leray spectral sequence with a mixed Hodge structure is via M.~Saito's theory of mixed Hodge modules, see~\cite[Remark 4.6(2)]{saito} and~\cite[Corollary 14.14]{mhs}. We expect our approach to agree with M.~Saito's.
\end{remark}

{\bf Acknowledgements:} We are grateful to Arshak Aivazyan for useful conversations, and to Semyon (Sam) Molokov for pointing out the reference~\cite{An16} to us.

\medskip

{\bf Organization:} The rest of the paper is organized as follows:
\begin{itemize}
    \item \textbf{Section \ref{pushouts_in_cat_of_schemes}:} We review pushouts in the category of schemes following Karl Schwede \cite{Sch05} and prove a few corollaries and modifications. The main results are Theorem~\ref{chapter1:main_theorem} and Corollary~\ref{chapter1:pushforward_of_finite_type}.
    \item \textbf{Section \ref{Construction_of_aff_model}:} We construct affine models, first assuming the base scheme $S$ is affine (Section~\ref{sec:main_thm_affine}) and then in general (Section~\ref{sec:main_thm_general}). The main result is Theorem \ref{chap2:main_thrm}. 
    \item \textbf{Section \ref{section_examples}:} We analyze Examples~\ref{ex:intro_spec_c}-\ref{ex:intro_cdh} and check that they indeed describe adequate classes of morphisms. For schemes over $\CC$ or a field of positive characteristic we additionally establish properties of affine models which are needed to prove the motivic nature of the Leray spectral sequence. The main results are Theorems~\ref{chap3:main_thm_Spec_C}, \ref{chap3:main_thrm_etale_case}, and \ref{chap3:A1_when_adequate}, which treat the cases of $S=\Spec (\CC)), S=\Spec(k), \chara k>0$, and $S$ arbitrary separarated Noetherian respectively.
    \item \textbf{Section \ref{Motivic_nature_of_Leray}:} We discuss the Leray spectral sequence of an arbitrary proper morphism of complex algebraic varieties. 
    After that, we briefly outline the analog in the \'etale case. The main results are Theorem \ref{chap4:main_thm} and Theorem \ref{chap4:etale_case_motivic_nature}, which cover the cases of $S=\Spec (\CC)$ and $S=\Spec(k),\chara k>0$ respectively.
    \item \textbf{Appendix:} The paper has an appendix (Appendix~\ref{appendix_Mayer_Vietoris}) in which we describe the Mayer-Vietoris exact sequence for sheaves on categories with a coverage. 
\end{itemize}

\section{Pushouts in the Category of Schemes}\label{pushouts_in_cat_of_schemes}

In this section we review pushouts in the category of schemes following mainly Schwede’s paper \cite{Sch05}, and prove a few modifications and generalizations along the way such as Lemma~\ref{chapter1:condition_variety_alg}, and its immediate geometric Corollary~\ref{chapter1:condition_variety_geom} which gives a sufficient condition for the resulting pushout to be a variety. We will need these later in order to construct affine models. All schemes in this section will be over a given base scheme $S$.

Let $\{X_i\}$ be a finite collection of schemes. Suppose that for each unordered pair of indices $\{i, j\}$ we are given a scheme $Z_{ij}$ and morphisms $\varphi_{\{i,j\},i}: Z_{ij} \rightarrow X_i$ and $\varphi_{\{i,j\},j}: Z_{ij} \rightarrow X_j$. This data gives rise to the diagram:
\[
\begin{tikzcd}
	& {X_i} \\
	{Z_{ij}} \\
	& {X_j}.
	\arrow["{\varphi_{\{i,j\},i}}", from=2-1, to=1-2]
	\arrow["{\varphi_{\{i,j\},j}}"', from=2-1, to=3-2]
\end{tikzcd}
\]
Putting together these diagrams for all $i,j$ we get a large diagram $D$ that contains all $X_i, Z_{ij}$ and all arrows $\varphi_{\{i,j\},i}$ and $\varphi_{\{i,j\},j}$.

\begin{theorem}\label{chapter1:main_theorem}
    If all morphisms in the diagram $D$ are closed embeddings, then the pushout exists in the category of schemes. If moreover all schemes are affine over $\Z$, then the pushout is isomorphic to the pushout in the category of affine schemes, and each $X_i$ embeds in the pushout as a closed subscheme. 
\end{theorem}
This is~\cite[Theorem 3.11]{Sch05}. 
The requirement that the arrows should be closed embeddings is necessary as the following example shows:
\begin{example}[Example 3.2 in~\cite{Sch05}]\label{chapter1:ex_push_not_affine}
    Consider the affine schemes $$ X \cong \Spec k[x,y,y^{-1}],\ Y \cong \Spec k[x,y]_{(x,y)},\mbox{ and } Z \cong \Spec(k[x,y]_{(x,y)}[y^{-1}]) $$ with the following natural maps:
    \[
    \begin{tikzcd}
    	Z \arrow[r,"{f^{\#}}"] \arrow[d,"{g^{\#}}"'] & X \\
    	Y
    \end{tikzcd}
    \]
    where \( f, g \) are the canonical inclusions
    \[
    f: k[x,y,y^{-1}] \hookrightarrow k[x,y]_{(x,y)}[y^{-1}],\quad g: k[x,y]_{(x,y)} \hookrightarrow k[x,y]_{(x,y)}[y^{-1}].
    \]
In this case the pushout of the diagram in the category of schemes does not exist.
\end{example}
Briefly, the idea of the proof of~\cite[Theorem 3.11]{Sch05} is to first construct the pushout in the category of ringed spaces (which is bicomplete) and then investigate whether or not the result is scheme. For example, the pushout of the diagram from Example~\ref{chapter1:ex_push_not_affine} in ringed spaces is not a scheme.

Suppose $X\gets Z\to Y$ is a diagram of schemes. The remainder of this section explores the question when the pushout ringed space, which we will denote $X\cup_Z Y$, is indeed a scheme.

\begin{lemma}\label{chapter1:affine_push_of_two}
Let $A, B$ be commutative rings, and $I \subset A$ an ideal. Let $\pi: A \to A/I$ be the canonical projection and $\gamma: B \to A/I$ a ring homomorphism. Set $Z = \Spec(A/I)$, $X = \Spec(A)$, and $Y = \Spec(B)$. 

Then the pushout $W = X \cup_Z Y$ of $X\gets Z\to Y$ in ringed spaces is the pushout of the same diagram in affine schemes, and also in all schemes.  Moreover, $W$ contains $Y$ as a closed subscheme.
\end{lemma}

This is~\cite[Theorems 3.4 and 3.5]{Sch05}. Note that the ring map $B\to A/I$ is not assumed surjective, so this lemma does not directly follow from Theorem~\ref{chapter1:main_theorem}.

\begin{example}[Example 3.7 in~\cite{Sch05}]
   Consider the following diagram
\[\begin{tikzcd}
	{\operatorname{Spec}k[x,y]/(x)} & {\operatorname{Spec}k[x,y]} \\
	{\operatorname{Spec}k}
	\arrow[from=1-1, to=1-2]
	\arrow[from=1-1, to=2-1]
\end{tikzcd}\]
where $k$ is a field.
Then by Lemma~\ref{chapter1:affine_push_of_two} the pushout in schemes exists and is affine. Calculating this pushout, denote it $W$, we get $W\cong \Spec k[x,xy, xy^{2}, \dots]$.
\end{example}

 This example shows that the pushout need not be variety even if $X,Y,Z$ are. We investigate the question whether the pushout is a variety in Corollary \ref{chapter1:condition_variety_geom}.

\begin{proposition}\label{chapter1:push_of_two}
Let $Z$ be a closed subscheme of quasi-compact schemes $X$ and $Y$. Then $X \cup_Z Y$ is the pushout in the category of schemes, and both $X$ and $Y$ are closed subschemes of $X \cup_Z Y$. 
\end{proposition}

This is~\cite[Corollary 3.9]{Sch05}.

In the sequel we will need to know whether or not the pushout is a variety. The following lemma give a sufficient condition for that.
\begin{lemma}\label{chapter1:condition_variety_alg}
    Let $R$ be a Noetherian commutative ring. Let $I_{1}, I_{2}, I_{3}$ be ideals in $R$. Consider the ideal $J = I_{1}+I_{2}+I_{3}$. Let $k_{1},\dots, k_{n}$ be a system of generators of $I_{3}$. Denote $A = R[x_{1},\dots, x_{n}]$. Consider the homomorphisms $\pi_{i}: A \rightarrow R/I_{i}$ the determined by formulas
    \begin{equation*}
        \pi_{1}(x_{j}) = k_{j},\ \pi_{2}(x_{j}) = 0.
    \end{equation*}
    
    Then the pullback of the diagram
\[\begin{tikzcd}
	& {R/I_{1}} \\
	{R/I_{2}} & {R/J}
	\arrow[from=1-2, to=2-2]
	\arrow[from=2-1, to=2-2]
\end{tikzcd}\]
in commutative rings is 
\begin{equation*}
    A/(Ker(\pi_{1})\cap Ker(\pi_{2})).
\end{equation*}
In particular, this pullback is a finitely generated $R$-algebra. 
\end{lemma}
\begin{proof}
     Note that the pullback of the diagram in the lemma is a ring $K$ such that for every pair $(a_{1}, a_{2}) \in R/I_{1}\times R/I_{2}$ with $a_{1} = a_{2}\ \operatorname{mod}\ J$ there exists a unique lift of that pair to $K$.\\
    First consider ring $A$. We claim that there exists a (not necessarily unique) lift of the pair $(a_{1}, a_{2})$ to $A$. Indeed, let $\bar{a}_{1}$ and $\bar{a}_{2}$ be lifts of $a_{i},i=1,2$ to $R$. We have $\bar{a}_{1} = \bar{a}_{2}$ modulo $J$, i.e.\ there exist $i_{j} \in I_{j}, j=1,2,3$ such that
    \begin{equation*}
        \bar{a}_{1}=\bar{a}_{2}+i_{1}+i_{2}+i_{3}.
    \end{equation*}
    Since $i_{3} \in I_{3} = (k_{1}, \dots, k_{n})$, we can rewrite the equation above in the form below:
    \begin{equation*}
        \bar{a}_{1} - i_{1} = \bar{a}_{2} + i_{2} + \sum\limits_{i=1}^{n}r_{i}k_{i}.
    \end{equation*}
    The desired lift is
    \begin{equation*}
        \bar{a}_{2}+i_{2}+\sum\limits_{i=1}^{n}r_{i}x_{i}.
    \end{equation*}
    In order to ensure the uniqueness of the lift it is sufficient to quotient $A$ by all elements that go to $0$ via $\pi_{1}\times \pi_{2}$, i.e.\ by $Ker(\pi_{1}) \cap Ker(\pi_{2})$.
\end{proof}
Dualising this lemma we get the following corollaries.
\begin{corollary}\label{chapter1:condition_variety_geom}
   Suppose $S=\Spec (k)$ where $k$ is an algebraically closed field. Let $X,Y, Z$ be $k$-varieties, and let $f:Z\to X, g:Z\to Y$ be closed embeddings. The pushout of the diagram
\[\begin{tikzcd}
	Z & Y \\
	X
	\arrow["g", hook, from=1-1, to=1-2]
	\arrow["f"', hook, from=1-1, to=2-1]
\end{tikzcd}\]
is again a variety.
\end{corollary}

\begin{corollary}\label{chapter1:pushforward_of_finite_type}
    Let $S$ be a Noetherian scheme and let
\[\begin{tikzcd}
	Z & Y \\
	X
	\arrow["g", hook, from=1-1, to=1-2]
	\arrow["f"', hook, from=1-1, to=2-1]
\end{tikzcd}\]
be a diagram of schemes of finite type over $S$. Suppose the maps $f$ and $g$ are closed embeddings. Then the pushout of this diagram is again a scheme of finite type over $S$.
\end{corollary}
\begin{proof}
    Since $S$ is Noetherian, it admits an open covering $\{ U_{i}\}_{i=1}^{n}$ with $U_{i}\cong \Spec R$ where $R$ is a Noetherian ring.

    Let $s: X\cup_{Z}Y\to S$ be the structure morphism induced by the structure morphisms $s_{X}: X\to S$, $s_{Y}: Y\to S$ and $s_{Z}: Z\to S$. Then the pullback $s^{-1}(U_{i})$ can be identified as follows
    \[
    s^{-1}(U_{i}) = s_{X}^{-1}(U_{i})\cup_{s_{Z}^{-1}(U_{i})} s_{Y}^{-1}(U_{i}).
    \]
    Without loss of generality one may assume that the pullbacks $s_{X}^{-1}(U_{i})$ and $s_{Z}^{-1}(U_{i})$ are affine. Then Lemma \ref{chapter1:condition_variety_alg} implies that the open set $s^{-1}(U_{i})$ is isomorphic to $\Spec T$ with $T$ being a finitely generated $R$-algebra. Hence the pushout is of finite type over $S$.
\end{proof}
\begin{remark}
    Unfortunately, the condition that $S$ be Noetherian is necessary in Corollary~\ref{chapter1:pushforward_of_finite_type}. This happens because for a non-Noetherian ring $R$ Lemma \ref{chapter1:condition_variety_alg} need not hold, as the following example shows.

    Let $R \cong k[x_{1},x_{2},\ldots]$, where $k$ is a field. Let $A \cong C \cong R$ and let $B \cong k[x_{1},x_{2},\ldots]/(x_{1},x_{2},\ldots) \cong k$. Let the maps $f:A \to B$ and $g: C \to B$ be the natural projections. Then the pullback $P$ is the $R$-algebra which consists of pairs of polynomials $(f_{1}(x_{1},x_{2},\ldots), f_{2}(y_{1},y_{2},\ldots))$ with the same constant term.

    We claim that this $R$-algebra $P$ is not finitely generated. Indeed, suppose the contrary, and let $\{ (a_{i}, b_{i}) \}_{i=1}^{n}$ be a finite set of generators of $P$. Take $m$ such that the generators $\{ (a_{i}, b_{i})\}_{i=1}^{n}$ do not contain the variable $x_{m}$.
    
    Since $\{ (a_{i}, b_{i})\}_{i=1}^{n}$ are generators of $P$, one can write
\[
(x_{m},0) = \sum\limits_{i=1}^{k}r_{i}h_{i},
\]
where $r_{i} \in R$ and $h_{i}$ are monomials in the generators $\{ (a_{i}, b_{i})\}_{i=1}^{n}$. Note that a term $r_{i}h_{i}$ is divisible by $x_{m}$ if and only if $r_{i}\in R$ is divisible by $x_{m}$. Hence one can rewrite the right hand side as follows
\[
x_{m}\cdot p + \text{term\ that\ are\ not\ divisible\ by\ } x_{m}
\]
where $p$ belongs to $P$.

It is easy to see that second sum is necessarily $0$. Hence we have an equality
\[
(x_{m},0) = x_{m}\cdot p,
\]
where $p \in P$. Therefore, dividing by $x_{m}$ we get $(1,0) = p \in P$, which is impossible.
\end{remark}

\section{Construction of an affine model}\label{Construction_of_aff_model}
In this section we construct affine models. (Recall that affine models and adequate classes of morphisms were defined in the the Introduction.) Unless stated otherwise, all schemes in this section will be assumed separated and of finite type over a base scheme $S$, which will be assumed Noetherian.

The main result of this section is as follows.
\begin{theorem}\label{chap2:main_thrm}
    Let $\mathcal{C}$ be the category of separated schemes of finite type over a separated Noetherian base scheme $S$. Then, for adequate class $\mathcal{W}$ of morphisms of $\mathcal{C}$, and any object $X \in \mathcal{C}$, there exists an affine model $m_{X}: M_{X} \to X$.

    Moreover, for any morphism $f: X\to Y$ in $\mathcal{C}$ there are affine models $m_{Y}: M_{Y}\to Y$ and $m_{X}: M_{X}\to X$, and a morphism $f_M:M_X\to M_Y$ such that the next diagram commutes:
\[\begin{tikzcd}
	{M_{X}} & {M_{Y}} \\
	X & Y.
	\arrow["{f_{M}}", from=1-1, to=1-2]
	\arrow["{m_{X}}"', from=1-1, to=2-1]
	\arrow["{m_{Y}}", from=1-2, to=2-2]
	\arrow["f"', from=2-1, to=2-2]
\end{tikzcd}\]
\end{theorem}
In the rest of the section we prove Theorem \ref{chap2:main_thrm}. The section consists of two parts. The first part deals with the case of an affine base scheme $S \cong \Spec R$. After proving the Theorem \ref{chap2:main_thrm} in this case we discuss basic properties of our construction. The second part contains the proof of Theorem \ref{chap2:main_thrm} in the general case.

\subsection{Affine base scheme}\label{sec:main_thm_affine}
As we mentioned in the Introduction, the main idea of the construction of an affine model is to turn open embeddings into close embeddings, and then to apply Theorem~\ref{chapter1:main_theorem} to obtain the pushout. The following lemmas shows how to turn an arbitrary morphism into a closed embedding.
\begin{lemma}\label{chap2:stacks_lemma}
    Let $S$ be a scheme. Let $g: X \to Y$ be a morphism of schemes over $S$. Consider the graph $i: X \to X\times_{S} Y$. If structure morphism $Y \to S$ is separated, then $i$ is a closed embedding.   
\end{lemma}
\begin{proof}
    See \cite[Lemma 67.4.6.]{Stacks}
\end{proof}
\begin{lemma}\label{chap2:product_clos_emb}
        Let $X, Y, Z$ be schemes over $S$ with the structure map $Y\to S$ separated. Suppose $f: X\to Y$ is an arbitrary morphism of schemes and $g: X\to Z$ is a closed embedding. Then the morphism
     \begin{equation*}
         g\times_S f: X \to Z\times_{S} Y
     \end{equation*}
     is a closed embedding.
\end{lemma}
\begin{proof}
    Note that the morphism $g\times_S f: X\to Z\times_{S}Y$ is the composition
    \[
    X \xrightarrow{i} X\times_{S} Y\xrightarrow{g\times_S \operatorname{id}} Z\times_{S}Y.
    \]
    The first morphism $i$ is a closed embedding because of Lemma \ref{chap2:stacks_lemma}. The second morphism is a closed embedding being the product of a closed embedding and the identity morphism. Hence $g\times_S f$ is a closed embedding too.
\end{proof}
\begin{lemma}\label{chap2:lemma_closed_affine_embeddings}
    Let $X$ be a scheme over $S$ such that $X\to\Spec(\Z)$ is affine. Assume that $S$ is separated and that the structure morphism $f: X \to S$ is of finite type. Then there exists a closed embedding
    \[
    X \hookrightarrow \mathbb{A}^{n}_{S}
    \]
    for some $n$.
\end{lemma}
\begin{proof}
    This is a particular case of~\cite[Lemma 70.13.1]{Stacks}.
\end{proof}
The proof of Theorem~\ref{chap2:main_thrm} in the case when $S$ is affine is contained in the next few lemmas. The first lemma deals with the special case when $X$ admits an affine open cover with only two elements.
\begin{lemma}\label{chap2:main_thrm_base_ind}
    Let $X = U_{1}\cup U_{2}$ be an open affine cover of a separated scheme $X$ of finite type over an affine Noetherian scheme $S$. Then there exists an affine model $m_{X}: M_{X} \to X$.
\end{lemma}
\begin{proof}
    Since the scheme $X$ is separated over $S$ and $S$ is affine, the intersection $U_{1} \cap U_{2}$ is affine over $S$ as well, so by Lemma~\ref{chap2:lemma_closed_affine_embeddings} there is a closed embedding $e: U_{1}\cap U_{2} \hookrightarrow \mathbb{A}^{n}_{S}$. Applying Lemma~\ref{chap2:product_clos_emb} one can see that all morphisms from the diagram below are closed embeddings
\[\begin{tikzcd}
	{U_{1}\cap U_{2}} & {U_{2}\times_{S}\mathbb{A}^{n}_{S}} \\
	{U_{1}\times_{S}\mathbb{A}^{n}_{S}}.
	\arrow["{\iota_{2}\times e}", hook, from=1-1, to=1-2]
	\arrow["{\iota_{1}\times e}"', hook, from=1-1, to=2-1]
\end{tikzcd}\]
Here $\iota_i:U_1\cap U_2\to U_i,i=1,2$ is the inclusion. Applying Theorem \ref{chapter1:main_theorem} we deduce that the pushout of this diagram in schemes over $S$, which we will denote $M_X$, exists and is affine (over $\Z$). By Corollary~\ref{chapter1:pushforward_of_finite_type}, $M_X$ is of finite type over $S$. The cover $\{ U_{i}\times \mathbb{A}^{n}_{S}\}_{i = 1,2}$ is a closed cover of $M_X$. The composition 
\[
U_{i}\times_{S}\mathbb{A}^{n}_{S}\to U_{i}\hookrightarrow X
\]
of the projection and the inclusion induces, by the universal property of a pushout, a morphism $m_{X}: M_{X}\to X$. Since the morphisms $U_{i}\times_{S} \mathbb{A}^{n}_{S} \to U_{i}, i=1,2$ belong to $\mathcal{W}$, and $\{ U_{i}\times \mathbb{A}^{n}_{S}\}_{i=1}^{2}$ is a closed cover of $M_{X}$, while $\{ U_{i}\}_{i=1}^{2}$ is an open cover of $X$, we can see that map $m_{X}: M_{X}\to X$ belongs to $\mathcal{W}$. Therefore $M_{X}$ is an affine model for $X$.
\end{proof}

We will need the following two lemmas.
\begin{lemma}\label{chap2:stability_of_pullback_for_2}
     Let $Y$ be a scheme of finite type over an affine Noetherian base scheme $S$. Assume there is an open affine cover $Y = U_{1} \cup U_{2}$. Let $f: X \to Y$ be an affine morphism of schemes over $S$ with $X$ separated over $S$. Let $M_{Y}\to Y$ be an affine model for $Y$ constructed as in the proof of Lemma \ref{chap2:main_thrm_base_ind}. Then the projection
     \[
      X\times_{Y}M_{Y}\to X
     \]
     belongs to $\mathcal{W}$.
\end{lemma}
\begin{proof}
    Recall from the proof of Lemma~\ref{chap2:main_thrm_base_ind} that we have an embedding $U_1\cap U_2\to\mathbb{A}^n_S$. Observe that the pullback $X\times_{Y}M_{Y}$ is the pushout of the following diagram
\[\begin{tikzcd}
	{f^{-1}(U_{1})\cap f^{-1}(U_{2})} & {f^{-1}(U_{2})\times_S\mathbb{A}^{n}_S} \\
	{f^{-1}(U_{1})\times_S \mathbb{A}^{n}_S},
	\arrow[from=1-1, to=1-2]
	\arrow[from=1-1, to=2-1]
\end{tikzcd}\]
and the projection $X\times_{Y} M_{Y}\to X$ is induced by the natural projections $f^{-1}(U_{i})\times_S\mathbb{A}^{n}_S\to f^{-1}(U_{i}), i=1,2$.
\end{proof}
\begin{lemma}\label{chap2:functoriality_aff_models_for_2}
    Let $Y$ be a scheme over an affine Noetherian base scheme $S$. Assume there is an open affine cover $Y = U_{1} \cup U_{2}$. Let $f: X \to Y$ be an affine morphism of schemes over $S$. Assume both $X$ and $Y$ are separated and of finite type over $S$. Then there are affine models $M_{X}\to X, M_{Y}\to Y$, and a morphism $f_{M}: M_{X} \to M_{Y}$ which make the following diagram commute:
\[\begin{tikzcd}
	{M_{X}} & {M_{Y}} \\
	X & Y.
	\arrow["{f_{M}}", from=1-1, to=1-2]
	\arrow["{m_{X}}"', from=1-1, to=2-1]
	\arrow["{m_{Y}}", from=1-2, to=2-2]
	\arrow["f"', from=2-1, to=2-2]
\end{tikzcd}\]
\end{lemma}
\begin{proof}
    Build $M_{Y}$ as in the proof of Lemma \ref{chap2:main_thrm_base_ind} using the cover $\{ U_{1}, U_{2}\}$. Let $e: U_{1}\cap U_{2} \hookrightarrow \mathbb{A}^{n}_{S}$ and $e_{X}: f^{-1}(U_{1})\cap f^{-1}(U_{2}) \hookrightarrow \mathbb{A}^{n}_{S}$ be closed embeddings. Consider the following diagram
\[\begin{tikzcd}
	{f^{-1}(U_{1})\cap f^{-1}(U_{2})} & {(f^{-1}(U_{2})\times_{S}\mathbb{A}^{n}_S)\times_{S}\mathbb{A}^{n}_S} \\
	{(f^{-1}(U_{1})\times_{S} \mathbb{A}^{n}_S)\times_{S}\mathbb{A}^{n}_S}.
	\arrow["{F_{2}\times_S e_{X}}", hook, from=1-1, to=1-2]
	\arrow["{F_{1}\times_S e_{X}}"', hook', from=1-1, to=2-1]
\end{tikzcd}\]
Here $F_{i}: f^{-1}(U_{1})\cap f^{-1}(U_{2}) \to f^{-1}(U_{i})\times_{S}\mathbb{A}^{n}_S, i=1,2$ is the product of the natural open inclusion and the following composition
\[
f^{-1}(U_{1})\cap f^{-1}(U_{2})\xrightarrow{f} U_{1}\cap U_{2} \xrightarrow{\iota_{i}\times_S e} U_{i}\times_{S}\mathbb{A}^{n}_{S}\xrightarrow{\operatorname{pr}_{2}} \mathbb{A}^{n}_S.
\]
 Using the same argument as in the proof of Lemma \ref{chap2:main_thrm_base_ind} one can see that the pushout $M_{X}$ of this diagram exists and is an affine model for $X$.

The morphism $f_{M}: M_{X} \to M_{Y}$ is induced by the following compositions:
\[
 (f^{-1}(U_{i})\times_{S}\mathbb{A}^{n}_S)\times_{S}\mathbb{A}^{n}_S \xrightarrow{\operatorname{pr}_{1}} f^{-1}(U_{i})\times_{S}\mathbb{A}^{n}_S \xrightarrow{f\times_S \operatorname{id}} U_{i}\times_{S}\mathbb{A}^{n}_S \hookrightarrow M_{Y}, i=1,2,
\]
and the universal property of a pushout.
\end{proof}

\begin{lemma}\label{lemma:main_ch_2_affine_case}
If $S$ is affine, then every finite type scheme $X$ over $S$ admits an affine model that can moreover be choses so that the pullback of it to any affine open subscheme $Y\to X$ is an affine model for $Y$.
\end{lemma}
\begin{proof} To prove the lemma it suffices to prove the following statements for every quasi-compact separated scheme $X$ over $S$:
\begin{itemize}
    \item There is an affine model $m_{X}: M_{X} \to X$.
    \item For any affine open inclusion $f: Y \to X$, the morphism $m_{Y}: M_{Y} = Y\times_{X}M_{X} \to Y$ is an affine model.
\end{itemize}

We will prove these by induction on the number of elements in an affine open cover. The base case of $X$ admitting an open cover by two affine subschemes is covered by Lemma \ref{chap2:main_thrm_base_ind} and Lemma \ref{chap2:stability_of_pullback_for_2}. Assume the above statements are true for every $X$ having an affine open cover with $n-1,n\geq 3$ elements, and let us prove the same statements for $X$ that has an affine open cover with $n$ elements.

Let $\{U_{i}\}_{i=1}^{n}$ be an open affine cover of $X$. 
Consider the subschemes $X' = U_{1}\cup\cdots\cup U_{n-1} \subset X$, and $Z = X'\cap U_{n}$. Then one has the following diagram
\begin{equation}\label{diag_pushout1}
\begin{tikzcd}
	Z & X' \\
	{U_{n}}
	\arrow[hook, from=1-1, to=1-2]
	\arrow[hook, from=1-1, to=2-1]
\end{tikzcd}
\end{equation}
where the maps are the natural inclusions.

Note that the schemes $X'$ and $Z$ have $n-1$-element affine coverings. So by the induction hypothesis, there are affine models $M_{X'}, M_{Z}$, and a morphism $M_{X'}\to M_Z$ such that diagram~(\ref{diag_pushout2}) maps into diagram~(\ref{diag_pushout1}):
\begin{equation}\label{diag_pushout2}
\begin{tikzcd}
	{M_{Z}} & {M_{X'}} \\
	{U_{n}}.
	\arrow[from=1-1, to=1-2]
	\arrow[from=1-1, to=2-1]
\end{tikzcd}
\end{equation}
Then, using the same argument as in proof of Lemma \ref{chap2:main_thrm_base_ind} an affine model for $X$ can be built as the pushout of the following diagram
\begin{equation}\label{diag_pushout3}
    \begin{tikzcd}
    	{M_{Z}} & {M_{X'}\times_{S}\mathbb{A}^{n}_S} \\
    	{U_{n}\times_{S}\mathbb{A}^{n}_S}.
    	\arrow[hook, from=1-1, to=1-2]
    	\arrow[hook, from=1-1, to=2-1]
    \end{tikzcd}
\end{equation}
Let $f: Y \to X$ be an affine open inclusion, and suppose $m_{X}: M_{X} \to X$ is the affine model for $X$ given by diagram~(\ref{diag_pushout3}). Then we claim that the natural projection 
\[
m_Y:Y\times_{X} M_{X} \to Y
\]
belongs to $\mathcal{W}$.

Indeed, since $f: Y \hookrightarrow X$ is an open embedding, the morphism $m_Y$ coincides with the morphism
\[
    m_{X}|_{m_{X}^{-1}(Y)}:m_{X}^{-1}(Y) \to Y.
\]

Recall that by diagram (\ref{diag_pushout3}) we have $$M_{X} = (M_{X'}\times_{S}\mathbb{A}^{n}_{S}) \cup_{M_{Z}} (U_{n}\times_{S}\mathbb{A}^{n}_{S}),$$ and the restrictions of $m_{X}$ to both $M_{X'}\times_{S}\mathbb{A}^{n}_{S}$ and $U_{n}\times_{S}\mathbb{A}^{n}_{S}$ belong to $\mathcal{W}$.

Then $m_{X}^{-1}(Y)$ can be covered as follows
\begin{equation*}
m_{X}^{-1}(Y) = (M_{X'\cap Y}\times_{S}\mathbb{A}^{n}_{S})\cup_{M_{Z\cap Y}}((U_{n}\cap Y)\times_{S}\mathbb{A}^{n}_{S}),
\end{equation*}
where $M_{X'\cap Y} = Y\times_X M_{X'}\to X'\cap Y$ and $M_{Z\cap Y} = Y\times_{X} M_{Z}\to Z\cap Y$ are the affine models obtained by pulling back affine models $M_{X'}\to X', M_Z\to Z$ and applying the induction hypothesis. So the restriction of $m_{Y}$ to each of the schemes $M_{X'\cap Y}\times_{S}\mathbb{A}^{n}_{S}, (U_{n}\cap Y)\times_{S}\mathbb{A}^{n}_{S}$ and $M_{Z \cap Y}$ is in $\mathcal{W}$.

Since $f$ is an affine morphism, so is the projection $\operatorname{pr}_{2}: Y\times_{X}M_{X} \to M_{X}$. So $Y\times_{X} M_{X}$ is affine over $\Z$. This concludes the induction step.

\end{proof}

To illustrate the construction we have just seen let us consider the following simple example.
\begin{example}
Consider the projective line \(\mathbb{P}^{1}_{k}\) over a field $k$ with the standard covering \(\mathbb{P}^{1}_{k} = \mathbb{A}^{1}_{(1)} \cup \mathbb{A}^{1}_{(2)}\) where \(\mathbb{A}^{1}_{(1)} = \operatorname{Spec}(k[x])\), \(\mathbb{A}^{1}_{(2)} = \operatorname{Spec}(k[y])\); the intersection is \(\operatorname{Spec}(k[z, z^{-1}])\).

The diagonal map  $\mathbb{A}^{1}_{(1)}\cap \mathbb{A}^{1}_{(2)} \to \mathbb{A}^{1}_{(1)}\times_k \mathbb{A}^{1}_{(2)}$ is induced by the homomorphism
\[
k[x,y]\to k[z,z^{-1}],
\]
which sends $x$ to $z$ and $y$ to $z^{-1}$. So we can identify $\mathbb{A}^{1}_{(1)} \cap \mathbb{A}^{1}_{(2)} \subset \mathbb{A}^{1}_{(1)} \times_{k}\mathbb{A}^{1}_{(2)} = \Spec(k[x,y])$ with $\Spec(k[x,y]/(xy-1))$.

Using Lemma~\ref{chap2:main_thrm_base_ind} and Lemma~\ref{chapter1:condition_variety_alg} we obtain an affine model for \(\mathbb{P}^{1}_{k}\) as the union of two $\mathbb{A}^3_k$'s glued along the quadric given by $xy=1$, or in coordinates 

\[ M_{\mathbb{P}^{1}_{k}} = \operatorname{Spec}\Bigl(k[x, y, z, w, t] / (w,t) \cap (w-z, t - (xy - 1)\Bigr).\]
This scheme can be covered via 
\[
U_{1} = \operatorname{Spec}(k[x,y,z,w,t]/(w,t)) \cong \operatorname{Spec}(k[x,y,z]) \cong \mathbb{A}^{3}_{k}
\]
and 
\[
U_{2} = \operatorname{Spec}(k[x,y,z,w,t]/(w - z, t - (xy - 1)) \cong \operatorname{Spec}(k[x,y,z]) \cong \mathbb{A}^{3}_{k}
\]
with the intersection
\[
U_{1}\cap U_{2} = \operatorname{Spec}(k[x,y,z,w,t]/(w,t,w-z,t-(xy-1)) \cong \operatorname{Spec}(k[x,y]/(xy-1)).
\]
Note that if $k = \CC$, by taking the complex points we get a space which is homotopy equivalent to the circle $S^{1}$. So $M_{\mathbb{P}_{\CC}^{1}}(\CC)$ is the union of two contractible spaces along a subspace homotopy equivalent to $S^1$ and hence is homotopy equivalent to the sphere $S^2$, as expected. 
\end{example}
\subsection{Arbitrary Noetherian base scheme}\label{sec:main_thm_general}
One can extract the following lemma from the proof of Lemma~\ref{lemma:main_ch_2_affine_case}.
\begin{lemma}\label{chap2:stab_of_pullback}
    Assume the base scheme $S$ affine and Noetherian. Then there is an affine model $m_{Y}: M_{Y} \to Y$ such that for an arbitrary morphism $f: X \to Y$ of separated $S$-schemes of finite type, the projection $X\times_{Y}M_{Y} \to X$ belongs to $\mathcal{W}$.
\end{lemma}
\begin{proof}
     Let $m_{Y}: M_{Y} \to Y$ be an affine model for $Y$ constructed as described above from an affine open cover $\{ U_{i}\}_{i=1}^{n}$ of $Y$.

    Consider the following pullback square
\[\begin{tikzcd}
	{X\times_{Y}M_{Y}} & {M_{Y}} \\
	X & Y\arrow[ul, phantom, "\ulcorner", very near end]. 
	\arrow["{\operatorname{pr}_{2}}", from=1-1, to=1-2]
	\arrow["{\operatorname{pr}_{1}}"', from=1-1, to=2-1]
	\arrow["{m_{Y}}", from=1-2, to=2-2]
	\arrow["f"', from=2-1, to=2-2]
\end{tikzcd}\]
Let $V_{i}$ be $\operatorname{pr}_{1}^{-1}\left(f^{-1}(U_{1})\cup\dots\cup f^{-1}(U_{i})\right)$. We want to prove that $\operatorname{pr}_{1}$ belongs to $\mathcal{W}$. To do this we will show that for any $2\leq i\leq n$ the restriction 
\[
\operatorname{pr}_{1}|_{V_{i}}: V_{i}\to f^{-1}(U_{1})\cup\cdots\cup f^{-1}(U_{i})
\]
belongs to $\mathcal{W}$. Substituting $i = n$ we will then obtain the desired result.

We proceed by induction. First, assume that $i = 2$. Then $M_{Y}$ is the pushout of the following diagram
\[\begin{tikzcd}
	{U_{1}\cap U_{2}} & {U_{2}\times_{S}\mathbb{A}^{n}_{S}} \\
	{U_{1}\times_{S}\mathbb{A}^{n}_{S}}.
	\arrow[hook, from=1-1, to=1-2]
	\arrow[hook, from=1-1, to=2-1]
\end{tikzcd}\]
The pullback $X\times_{Y} M_{Y}$ can be identified with the pushout of the next diagram:
\[\begin{tikzcd}
	{f^{-1}(U_{1})\cap f^{-1}(U_{2})} & {f^{-1}(U_{2})\times_{S}\mathbb{A}^{n}_{S}} \\
	{f^{-1}(U_{1})\times_{S}\mathbb{A}^{n}_{S}}.
	\arrow[hook, from=1-1, to=1-2]
	\arrow[hook, from=1-1, to=2-1]
\end{tikzcd}\]

Since the restrictions of $\mathrm{pr}_{1}$ to $f^{-1}(U_{1})\times_{S} \mathbb{A}^{n}_{S}$ and $f^{-1}(U_{2})\times_{S}\mathbb{A}^{n}_{S}$ are the projections to the first factor, while the restriction to $f^{-1}(U_{1}) \cap f^{-1}(U_{2})$ is the identity, the resulting morphism $\mathrm{pr}_{1}|_{V_{2}} \to f^{-1}(U_{1}) \cup f^{-1}(U_{2})$ belongs to $\mathcal{W}$.


Now we perform the induction step. Assume that $\operatorname{pr}_{1}|_{V_{i}}$ belongs to $\mathcal{W}$. We want to show that the same holds for the restriction of $\operatorname{pr}_{1}$ to $V_{i+1}$. To see this recall that $M_{U_{1}\cup \cdots \cup U_{i+1}}$ is the pushout of the following diagram
\[\begin{tikzcd}
	{M_{1}} & {M_{2}\times_{S}\mathbb{A}^{n}_{S}} \\
	{U_{i+1}\times_{S}\mathbb{A}^{n}_{S}}
	\arrow[hook, from=1-1, to=1-2]
	\arrow[hook, from=1-1, to=2-1]
\end{tikzcd}\]
where $M_{1}$ is an affine model for $(U_{1}\cup \cdots\cup U_{i})\cap U_{i+1}$, and $M_{2}$ is an affine model for $U_{1}\cup \cdots \cup U_{i}$. 

So the pullback $X\times_{Y} M_{Y}$ can be identified with the pushout of
\[\begin{tikzcd}
	{X\times_{Y}M_{1}} & {(X\times_{Y}M_{2})\times_{S}\mathbb{A}^{n}_{S}} \\
	{f^{-1}(U_{i+1})\times_{S}\mathbb{A}^{n}_{S}}
	\arrow[hook, from=1-1, to=1-2]
	\arrow[hook, from=1-1, to=2-1]
\end{tikzcd}\]
where the restrictions of $\operatorname{pr}_{1}$ to $X\times_{Y} M_{1}$ and $X\times_{Y} M_{2}$ belong to $\mathcal{W}$ by the induction hypothesis. Since the restriction of $\operatorname{pr}_{1}$ to $f^{-1}(U_{i+1})\times_{S}\mathbb{A}^{n}_{S}$ coincides with the projection 
\[
f^{-1}(U_{i+1})\times_{S}\mathbb{A}^{n}_S \to f^{-1}(U_{i+1}),
\]
we deduce that the restriction of $\operatorname{pr}_{1}$ to $V_{i+1}$ belongs to $\mathcal{W}$.
\end{proof}

Note that the construction of an affine model is non-functorial. However, in the proof of the affine base of Theorem \ref{chap2:main_thrm} we used several times the fact that it is functorial with respect to affine inclusions. The following lemma generalizes this result to the case of an arbitrary morphism $f$, and it will play a crucial role in the proof of~Theorem~\ref{chap2:main_thrm} over an arbitrary quasi-compact separated base scheme. 
\begin{lemma}\label{chap2:semi-functoriality}
    Assume the base scheme $S$ is affine and Noetherian. Let $f: X\to Y$ be an arbitrary morphism in $\mathcal{C}$ where $\mathcal{C}$ is as in Theorem~\ref{chap2:main_thrm}. Then there exist affine models $m_{X}: M_{X} \to X,\ m_{Y}: M_{Y}\to Y$ together with a morphism $f_{M}: M_{X} \to M_{Y}$ such that the following diagram commutes:
\[\begin{tikzcd}
	{M_{X}} & {M_{Y}} \\
	X & Y.
	\arrow["{f_{M}}", from=1-1, to=1-2]
	\arrow["{m_{X}}"', from=1-1, to=2-1]
	\arrow["{m_{Y}}", from=1-2, to=2-2]
	\arrow["f"', from=2-1, to=2-2]
\end{tikzcd}\]
Moreover, $M_{Y}$ does not depend on $X$ and $f$.
\end{lemma}
\begin{proof}
   By Lemma \ref{chap2:stab_of_pullback}, the morphism
   \[
   \operatorname{pr}_{1}: X\times_{Y}M_{Y} \to X
   \]
   belongs to $\mathcal{W}$. Let $M_{X}$ be an affine model for $X\times_{Y}M_{Y}$. Then the natural projection to $M_X\to X$ belongs to $\mathcal{W}$ being the composition of two morphisms from $\mathcal{W}$. The morphism $f_{M}:M_{X} \to M_{Y}$ is defined as the following composition
   \[
   M_{X} \to X\times_{Y}M_{Y} \xrightarrow{\operatorname{pr}_{2}} M_{Y}.
   \]
\end{proof}

\begin{proof}[Proof of Theorem \ref{chap2:main_thrm}]
Since $S$ is quasi-compact it admits a finite open affine covering $\{ U_{i}\}_{i=1}^{n}$. 
We will proceed by induction on the number $n$ of elements in the covering of $S$. If $n = 1$, then Theorem \ref{chap2:main_thrm} follows from Lemmas~\ref{lemma:main_ch_2_affine_case} and~\ref{chap2:semi-functoriality}.

Next let us describe the induction step. Assume that we are able to build an affine model for every separated scheme of finite type over a base scheme which has an open affine cover of at most $n-1, n\geq 2$ elements. Moreover, assume that for any morphism $f: X \to Y$ of such schemes, there are affine models $m_{X}: M_{X} \to X,\ m_{Y}: M_{Y} \to Y$ and a morphism $f_M:M_X\to M_Y$ making the obvious square commute.

Let $X\in\mathcal{C}$, and let $F:X\to S$ be the structure morphism. By Lemma \ref{chap2:lemma_closed_affine_embeddings} any affine scheme over $S$ of finite type can be embedded in some $\mathbb{A}^{n}_{S}$ as a closed subscheme. Therefore we can construct the following diagram 
\[\begin{tikzcd}
	{M_{1}} & {M_{2}\times_{S}\mathbb{A}^{n}_{S}} \\
	{M_{3}\times_{S}\mathbb{A}^{n}_{S}}
	\arrow[hook, from=1-1, to=1-2]
	\arrow[hook, from=1-1, to=2-1]
\end{tikzcd}\]
where $M_{2}$ is an affine model for $F^{-1}(U_{1})\cup\cdots \cup F^{-1}(U_{n-1})$, $M_{1}$ is an affine model for $(F^{-1}(U_{1})\cup \cdots \cup F^{-1}(U_{n-1}))\cap F^{-1}(U_{n})$, and $M_{3}$ is an affine model for $F^{-1}(U_{n})$. Then the pushout of that diagram will be an affine model for $F^{-1}(U_{1})\cup\cdots \cup F^{-1}(U_{n})$.

The rest of the proof, namely the proof that every $f\in \mathcal{C}(X,Y)$ can be covered by a morphism of appropriate affine models, is similar to the proof of Lemma~\ref{chap2:semi-functoriality}.
\end{proof}

\section{Examples of adequate classes of morphisms}\label{section_examples}
In this section we check that the classes of morphisms given in Examples~\ref{ex:intro_spec_c}-\ref{ex:intro_cdh} are adequate. Let $S$ be a separated Noetherian scheme.

The first subsection contains the proof of the Theorem \ref{chap3:A1_when_adequate}, which states that after choosing an appropriate topology on the category of separated schemes of finite type over $S$ the weak equivalences in the resulting $\mathbb{A}^{1}$-homotopy theory form an adequate class of morphisms.

The second, respectively third subsection studies very particular cases of $S$, namely $S = \Spec \mathbb{C}$, respectively $S=\Spec(k)$ where $k$ is a field of characteristic $p>0$. After checking the axioms we will prove some additional properties of affine models which we will need in order to establish the motivic nature of the Leray spectral sequence.

\subsection{\texorpdfstring{$\mathbb{A}^{1}$}{}-homotopy theory}\label{example_cdh}
Here we will discuss Example~\ref{ex:intro_cdh}. More specifically, we review, following~\cite{An16}, the basic constructions of unstable $\mathbb{A}^{1}$-homotopy theory and show that, under an appropriate choice of a Grothendieck topology on the category of schemes of finite type over $S$, $\mathbb{A}^{1}$-equivalences form an adequate class of morphisms. 

If $\mathcal{C}',\mathcal{C}''$ are categories, we denote the category of functors $\mathcal{C}'\to \mathcal{C}''$ by $[\mathcal{C}',\mathcal{C}'']$. Let $\mathcal{C}$ be the category of separated schemes of finite type over a base scheme $S$. We use $\operatorname{Set}$, respectively $\operatorname{sSet}$ to denote the category of sets, respectively of simplicial sets. 

Consider the category of simplicial presheaves $[\mathcal{C}^{op}, \operatorname{sSet}]$ on $\mathcal{C}$, i.e.\ the category of contravariant functors from $\mathcal{C}$ to $\operatorname{sSet}$. Note that each scheme $X \in \mathcal{C}$ gives us the constant simplicial presheaf defined by $X(U)_{n} = \operatorname{Hom}(U, X)$.
This correspondence defines an embedding of categories $\mathcal{C} \hookrightarrow [\mathcal{C}, \operatorname{sSet}]$.

Next, we need a model structure on $[\mathcal{C}^{op}, \operatorname{sSet}]$. In the sequel we will use the standard (Quillen) model structure on $\operatorname{sSet}$. So we get the \emph{projective} model structure on $[\mathcal{C}^{op}, \operatorname{sSet}]$ in which
a morphism of presheaves $\mathcal{F} \to \mathcal{G}$ is a weak equivalence (respectively a fibration) if, for any object $X \in \mathcal{C}$, the corresponding morphism
$\mathcal{F}(X) \to \mathcal{G}(X)$
is a weak equivalence (respectively a fibration) of simplicial sets. The cofibrations are defined using the left lifting property with respect to the acyclic fibrations. The result is indeed a model structure, see~\cite[Proposition 3.32]{An16} and the references in the proof.

The next step involves enlarging the class of weak equivalences by taking the left Bousfield localization with respect to a class of morphisms constructed using a Grothendieck topology. Let us mention a few examples of topologies on $\mathcal{C}$ used in motivic homotopy theory.
\begin{example}\label{chap3:example_Nis_top}
  The \emph{Nisnevich topology} \cite[Section 1.1]{Ni89} is generated by collections of \'etale maps
    \[
   \{U_{i}\}_{i \in I} \to X,
    \]
    such that the following conditions are satisfied:
    \begin{itemize}
        \item For any point $x \in X$ there is an $i \in I$, and a point $u_{i} \in U_{i}$ over $x$ such that the corresponding morphism of the residue fields is an isomorphism.
    \end{itemize}

\end{example}
\begin{example}\label{chap3:example_chd_top}
    The \emph{cdh topology}, introduced in \cite{Voe09}, is generated by the Nisnevich topology and the covers of the following type
    \[
    \{\{ A\xrightarrow{e} X\},\ \{Y\xrightarrow{p} X\} \},
    \]
    where $e$ is a closed embedding and the morphism $p$ is proper, and moreover
    \[
    p^{-1}(X- A)\to X- A
    \]
    induces an isomorphism of the reduced schemes.
\end{example}
\begin{example}\label{chap3:example_open_closed_top}
    Our final example is the \emph{open-closed topology}. It was introduced and used by Kahn in~\cite{Ka02}, and it is generated by finite closed and finite open covers.
\end{example}
In order to define the Bousfield localization of $[\mathcal{C}^{op}, \operatorname{sSet}]$ that we are after we will need hypercovers. Recall that one can view $V\in \mathcal{C}$ as a simplicial presheaf via the embedding $\mathcal{C}\to[\mathcal{C}^{op},\operatorname{Set}]\to[\mathcal{C}^{op},\operatorname{sSet}]$. A $\tau$-hypercover of $V\in \mathcal{C}$ is a map of simplicial presheaves $U\to V$ of a particular kind, constructed using $\tau$, see e.g.~\cite[Section~3.4]{An16}. We will mainly be interested in $\tau$-hypercovers of the simplest possible type, namely \v{C}ech hypercovers, which we will now describe.
\begin{example}\label{chap3:example_cech_object}
    Let $U \to V$ be a $\tau$-cover of $V$. Let us introduce a simplicial object in $\mathcal{C}$ called a \emph{\v{C}ech object} and defined by
    \[
     \check{U}_{n} = \overbrace{U\times_{V}U\times_{V}\dots\times_{V}U}^{n+1\ \text{elements}}.
    \]
   The face and degeneracy morphisms, and the morphism $\check{U}_\bullet\to V$ are defined in a straightforward way. (Here $V$ is viewed as a constant simplicial object.)

    One can consider $\check{U}_{\bullet}$ as a simplicial presheaf $\check{U}$ on $\mathcal{C}$ via $$[\triangle^{op},\mathcal{C}]\to [\triangle^{op},[\mathcal{C}^{op},\operatorname{Set}]]=[\mathcal{C}^{op},[\triangle^{op},\operatorname{Set}]]=[\mathcal{C}^{op},\operatorname{sSet}].$$ The resulting map $\check{U}\to V$ of simplicial presheaves is an example of a $\tau$-hypercover, called a \emph{\v{C}ech hypercover}. 
\end{example}
Using~\cite[Proposition 3.28]{An16} we get a model structure on $[\mathcal{C}^{op}, \operatorname{sSet}]$ by taking the left Bousfield localization of the projective model structure with respect to the following morphisms:
\begin{itemize}
    \item the projections $X\times_{S}\mathbb{A}^{1}_{S} \to X$;
    \item $U_{\bullet}\to V$ with $U_{\bullet}$ being a $\tau$-hypercover of $V$.
\end{itemize}
(Recall that $\mathcal{C}$ is essentially small, so we may assume that these morphisms form a set, and~\cite[Proposition 3.28]{An16} is applicable.) Let $\mathrm{L}_{\tau,\mathbb{A}^1}[\mathcal{C}^{op}, \operatorname{sSet}]$ denote the resulting model category.
\begin{theorem}\label{chap3:A1_when_adequate}
    Assume that the topology $\tau$ contains both finite (Zariski-)open covers and finite (Zariski-)closed covers. Let $\mathcal{W}$ be the class of morphisms in the category $\mathcal{C}$ which go to weak equivalences of $\mathrm{L}_{\tau,\mathbb{A}^1}[\mathcal{C}^{op}, \operatorname{sSet}]$ under the embedding $\mathcal{C} \hookrightarrow [\mathcal{C}^{op}, \operatorname{sSet}]$. Then the class $\mathcal{W}$ is adequate. 
\end{theorem}
\begin{example}\label{chap3:example_admisssible_topologies}
    This theorem applies to $\tau$ equal the cdh topology or the open-closed topology discussed in Examples \ref{chap3:example_chd_top} and \ref{chap3:example_open_closed_top} respectively.
\end{example}
Before proving Theorem \ref{chap3:A1_when_adequate}, we will review a few facts about localizations and derived functors. If $\mathcal{M}$ is a category with weak equivalences, we let $\Ho(\mathcal{M})$ denote the localization in the weak equivalences.

\begin{definition}
     Let $\mathcal{M}$ be a category with weak equivalences and $\mathcal{N}$ an arbitrary category. Let $\gamma:\mathcal{M}\to\Ho(\mathcal{M})$ be the localization functor. Suppose $F: \mathcal{M} \to \mathcal{N}$ is a functor. Then the left derived functor $LF$ of $F$ can be defined as the left universal functor $LF:\Ho(\mathcal{M})\to\mathcal{N}$ such that there is a natural transformation $LF\circ\gamma\to F$. Left universal means here that all other such functors $\Ho(\mathcal{M})\to\mathcal{N}$ map into $LF$ and the obvious diagrams commute. The details can be found e.g.\ in~\cite[Section 9]{DwSpa}.
\end{definition}
\begin{remark}
 The left derived functor may or may nor exist. However if it does exist, then it is unique up to isomorphism. Suppose $\begin{tikzcd}\mathcal{M}\ar[r,bend left,"l"]& \mathcal{N}\ar[l,bend left,"r"]\end{tikzcd}$ is a Quillen adjunction, and $F$ is the composition of the left Quillen adjoint $r:\mathcal{M}\to\mathcal{N}$ and the localization $\mathcal{N}\to \Ho(\mathcal{N})$. Then the left derived functor of $F$ exists and can be calculated on cofibrant objects of $\mathcal{M}$ (see e.g.~\cite[Proposition 9.3]{DwSpa}). 
\end{remark}

Let now $I$ be a small category and $\mathcal{M}$ a category with weak equivalences. We define the weak equivalences in $[I,\mathcal{M}]$ object-wise. Assume that we have an adjunction
\begin{equation}\label{eq:colim_const_adj}
    \colim: 
    \begin{tikzcd} 
    [][I, \mathcal{M}] \ar[r,bend left]& \mathcal{M}\ar[l,bend left]
    \end{tikzcd}
    :\const
\end{equation}
where $\const$ is the constant diagram functor. Assuming in addition that the left derived functor $L\colim$ exists, we define the \emph{homotopy colimit functor} $\hocolim: \Ho([I, \mathcal{M}]) \to \Ho (\mathcal{M})$ as the composition of $L\colim$ and the localization functor $\gamma:\mathcal{M} \to \Ho (\mathcal{M})$. Note that $L\colim$ does exist e.g.\ if $\mathcal{M}$ is a model category such that the projective model structure on $[I,\mathcal{M}]$ exists as then~(\ref{eq:colim_const_adj}) becomes a Quillen adjunction.

\smallskip

Let $F\in[\mathcal{C},\operatorname{sSet}]$ be a simplicial presheaf. One can view $F$ as a simplicial object in simplicial presheaves, denoted $F_\bullet$, via
$$[\mathcal{C}^{op},[\triangle^{op},\operatorname{Set}]]=[\triangle^{op},[\mathcal{C},\operatorname{Set}]]\to [\triangle^{op},[\mathcal{C},\operatorname{sSet}]].$$

The homotopy colimit functor
$$\hocolim_{\triangle^{op}}:\Ho ([\triangle^{op},[\mathcal{C},\operatorname{sSet}]])\to \Ho([\mathcal{C},\operatorname{sSet}]])$$ exists (as can be seen e.g.\ by applying Proposition 4.3 of~\cite{An16}). Moreover, there is a functor 
$\mathop{diag}:[\triangle^{op},[\mathcal{C},\operatorname{sSet}]]\to [\mathcal{C},\operatorname{sSet}]]$ on the level of the non-homotopy categories which calculates $\hocolim_{\triangle^{op}}$:
\begin{lemma}\label{lemma:hocolim_simpl_presheaves}
Let 
$$\mathop{diag}:[\triangle^{op},[\mathcal{C}^{op},\operatorname{sSet}]]=[\triangle^{op}\times\mathcal{C}^{op}\times\triangle^{op},\operatorname{Set}]\to [\mathcal{C}^{op}\times\triangle^{op},\operatorname{Set}]=[\mathcal{C},\operatorname{sSet}]$$ be induced by the diagonal embedding $\triangle^{op}\to\triangle^{op}\times\triangle^{op}$. Then the following diagram commutes up to isomorphism of functors:
$$
\begin{tikzcd}[column sep=huge]
\left[\triangle^{op},\left[\mathcal{C}^{op},\operatorname{sSet}\right]\right]\arrow[r,"\mathop{diag}"]\arrow[d] & \left[\mathcal{C}^{op},\operatorname{sSet}\right]\arrow[d]\\
\Ho\left[\triangle^{op},\left[\mathcal{C}^{op},\operatorname{sSet}\right]\right]\arrow[r,"\hocolim_{\triangle^{op}}"] & \Ho\left[\mathcal{C}^{op},\operatorname{sSet}\right].
\end{tikzcd}
$$

Moreover, if $F_\bullet\in [\triangle^{op},[\mathcal{C},\operatorname{sSet}]]$ is obtained from $F\in [\mathcal{C},\operatorname{sSet}]$, then $\mathop{diag}(F_\bullet)\cong F$, functorially in $F$.
\end{lemma}

\begin{proof} The first assertion is~\cite[Remark 2.1]{dhi}. The second one is straightforward.\end{proof}

\begin{proof}[Proof of Theorem \ref{chap3:A1_when_adequate}]
    It is clear from the definitions that the class of morphisms $\mathcal{W}$ defined in the theorem is closed under composition and contains the projections $X\times_{S}\mathbb{A}^{1}_{S} \to X$, so the only non-trivial condition is the last one. Let $\{ U_{i}\}_{i=1}^{2}$ and $\{ V_{i}\}_{i=1}^{2}$ be a closed, respectively an open cover of the scheme $X$, respectively $Y$. Let $f: X\to Y$ be a morphism such that $f(U_i)\subset V_i$ for $i=1,2$, and assume that the morphisms $U_1\to V_1, U_2\to V_2, U_1\cap U_2\to V_1\cap V_2$ induced by $f$ belong to $\mathcal{W}$. We need to show that then so does $f$.

    Let $\check{U}$ and $\check{V}$ be the \v{C}ech objects associated with the covers $\{ U_{i}\}_{i=1}^{2}$ and $\{ V_{i}\}_{i=1}^{2}$ respectively and viewed as objects of $[\mathcal{C}^{op}, \operatorname{sSet}]$. The morphism $f$ induces a morphism $\check{f}: \check{U}\to \check{V}$. Moreover, viewing $\check{U}$ and $\check{V}$ as objects $\check{U}_\bullet,\check{V}_\bullet\in[\triangle^{op},[\mathcal{C},\operatorname{sSet}]]$ and applying Lemma~\ref{lemma:hocolim_simpl_presheaves} we get the commutative diagram in $[\triangle^{op},[\mathcal{C},\operatorname{sSet}]]$
$$
\begin{tikzcd}
\mathop{diag}(\check{U}_\bullet) \ar[r,"\cong"] \ar[d,"\sim_{\mathrm{L}}","\mathop{diag}(\check{f}_\bullet)"'] & {\check{U}} \ar[r,"\sim_{\mathrm{L}}"]\ar[d,"\check{f}"] & X \ar[d,"f"]\\
\mathop{diag}(\check{V}_\bullet) \ar[r,"\cong"] & {\check{V}} \ar[r,"\sim_{\mathrm{L}}"] & Y.
\end{tikzcd}
$$

The arrows marked $\cong$ are isomorphisms by Lemma~\ref{lemma:hocolim_simpl_presheaves}. The arrows marked $\sim_{\mathrm{L}}$ become weak equivalences after we left Bousfield localize: for $\check{U}\to X$ and $\check{V}\to Y$ this follows from the definition of the localization, and $\mathop{diag}(\check{f}_\bullet)$ represents $\hocolim(\check{f}_\bullet)$ by Lemma~\ref{lemma:hocolim_simpl_presheaves}. We conclude that $f$ is a weak equivalence in $\mathrm{L}_{\tau,\mathbb{A}^1}[\mathcal{C}^{op}, \operatorname{sSet}]$, i.e.\ $f\in\mathcal{W}$.
\end{proof}
\subsection{Schemes over \texorpdfstring{$\CC$}{}}\label{example_spec_c}
In this section we investigate Example~\ref{ex:intro_spec_c}. Let $\mathcal{C}$ be the category of separated schemes of finite type over $\Spec \mathbb{C}$. In Section~\ref{example_spec_c}, we will often identify $X\in\mathcal{C}$ with $X(\CC)$, always assumed equipped with the complex analytic topology. In particular, by saying that $X\in\mathcal{C}$ is contractible we will mean that $X(\CC)$ is contractible, by a sheaf on $X\in\mathcal{C}$ we will mean a sheaf on $X(\CC)$ etc. 
For simplicity all complexes of sheaves are assumed bounded below. We expect this restriction can be removed if need be.

We will show in Proposition \ref{chap3:admisibilty} that $\mathcal{C}$ together with the set of morphisms $\mathcal{W}$ that induce a homotopy equivalence on the complex points is an adequate pair. 

Moreover, we will prove the following theorem. 
\begin{theorem}\label{chap3:main_thm_Spec_C}
    Let $X$ be a separated scheme of finite type over $\Spec \mathbb{C}$. Then $X$ admits an affine model $m_{X}: M_{X} \to X$ that has the following properties:
    \begin{enumerate}
        \item\label{item:ho_eq} $m_{X}$ is a homotopy equivalence.
        \item $m_{X}$ has contractible fibers.
        \item\label{item:coho_iso} For every bounded below complex of sheaves $\mathcal{F}^{\bullet}$ on $X$, the natural map 
        \[
            H^{*}(X; \mathcal{F}^{\bullet}) \to H^{*}(M_X; m_{X}^{*}(\mathcal{F}^{\bullet}))
        \]
        is an isomorphism.
        \item\label{item:base_change}
        If $X'\to X$ is an arbitrary morphism of separated schemes of finite type over $\CC$, then the base change map along $m_X:M_X\to X$ is an isomorphism for any bounded below complex of sheaves on $X'$.
    \end{enumerate}
\end{theorem}
This theorem will be needed in the proof that the Leray spectral sequence of a proper morphism is motivic as it will allow us to compare the Leray spectral sequence for $X \to Y$ with the Leray spectral sequence of the base change to an affine model, i.e.\ to the Leray spectral sequence of the morphism $X\times_{Y} M_{Y} \to M_{Y}$. The last part of the theorem will be used in~\cite{egor_non_proper} to prove that the Leray spectral sequence of an arbitrary (i.e.\ not necessarily proper) morphism in $\mathcal{C}$ is also motivic.

\subsubsection{Proof of Theorem~\ref{chap3:main_thm_Spec_C}, parts~\ref{item:ho_eq}-\ref{item:coho_iso}}

\begin{proposition}\label{chap3:admisibilty}
    The pair $(\mathcal{C}, \mathcal{W})$ defined above is adequate.
\end{proposition}
\begin{proof}
The only non-trivial condition on $\mathcal{W}$ that needs checking is the following:
\begin{itemize}
    \item Let $f: X\to Y$ be a morphism in $\mathcal{C}$ which is compatible with a closed covering $X = X_{1}\cup X_{2}$ and an open covering $Y = Y_{1}\cup Y_{2}$, meaning that $f(X_i)\subset Y_i, i=1,2$. If the restrictions
    \[
    f|_{X_{i}}: X_{i}\to Y_{i},\ f|_{X_{1}\cap X_{2}}: X_{1}\cap X_{2} \to Y_{1} \cap Y_{2}
    \]
    are homotopy equivalences, then so is $f: X\to Y$.
\end{itemize}

By Theorem \ref{appendix:Mayer_Vietoris_main_thrm} there is a Mayer-Vietoris sequence both for the closed cover $X=X_1\cup X_2$ and for the open cover $Y=Y_1\cap Y_2$, and $f$ induces a natural map of these sequences. Next, note that a homotopy equivalence $g: Z \rightarrow W$ induces, for every locally constant sheaf $\mathcal{G}$ on $W$, an isomorphism
\[
g^{*}: H^{*}(W; \mathcal{G}) \to H^{*}(Z; g^{*}\mathcal{G}).
\]
Applying the Mayer-Vietoris sequences and the $5$-lemma we see that the morphism $f: X\to Y$ induces an isomorphism of the cohomology groups
\[
f^{*}: H^{*}(Y; \mathcal{F})\to H^{*}(X; f^{*}\mathcal{F})
\]
for every locally constant sheaf $\mathcal{F}$ on $Y$. In particular, $f$ induces a bijection $\pi_{0}(X)\to \pi_{0}(Y)$.

Let us choose $A \subset X$ to be a subset such that $A$ contains at least one element in each path component of $X_{1}\cap X_{2}, X_{1}$ and $X_{2}$; and $f|_A$ is injective. 

By our assumptions, $f$ induces isomorphisms of the fundamental groupoids 
$$
\begin{aligned}
\pi_1(X_{1}\cap X_{2},A\cap X_1\cap X_2)&
\to \pi_1(Y_{1}\cap Y_{2},f(A\cap X_1\cap X_2)),\\
\pi_1(X_{1},A\cap X_1)&\to \pi_1(Y_{1},f(A\cap X_1)),\\
\pi_1(X_{2},A\cap X_2)&\to \pi_1(Y_{2},f(A\cap X_2)).
\end{aligned}
$$
Applying the groupoid pushout version of the Seifert-Van Kampen theorem (see e.g.~\cite[6.7.2]{Brown}), we see that $f_*:\pi_1(X,A)\to \pi_1(Y,f(A))$ is an isomorphism of groupoids.

So $f$ induces isomorphisms of the fundamental groups of the path components, and an isomorphism of the cohomology groups of arbitrary local systems. This implies that $f$ is a homotopy equivalence, cf.\ \cite[Section 4.2, Exercise 12]{Hatcher_AT}.
\end{proof}

We have now proved that the pair $(\mathcal{C}, \mathcal{W})$ is adequate. So we can apply Theorem~\ref{chap2:main_thrm} to obtain an affine model $m_{X}: M_{X} \to X$. Combining this with Corollary~\ref{chapter1:condition_variety_geom} and Lemma~\ref{chap2:stab_of_pullback}, we get the following result.
\begin{theorem}\label{chap3:pullback_along_subvariety}
    Every scheme $X\in\mathcal{C}$ admits an affine model $m_{X}: M_{X} \to X$. If $X$ is a variety (i.e.\ if $X$ is reduced), then $M_X$ is again a variety. Moreover, for every morphism $Y \subset X$ in $\mathcal{C}$ the pullback $Y\times_X M_X \to Y$
    belongs to $\mathcal{W}$.
\end{theorem}
An immediate corollary of this theorem is
\begin{corollary}\label{chap3:contractible_fibers}
    The complex points of the fibers of $m_{X}$ are contractible (in the complex analytic topology).
\end{corollary}
Before diving into the proof of the main theorem of this section we need the following preliminary lemma.
\begin{lemma}\label{chap3:simple_projection_induce_iso}
  Let $X,Y$ be topological spaces with $Y$ contractible. Then for every sheaf $\mathcal{F}$ on $X$, the map $\pi: X\times Y \rightarrow X$
    induces an isomorphism
    \begin{equation*}
        \pi^{*}: H^{*}(X; \mathcal{F})\rightarrow H^{*}(X\times Y; \pi^{*}\mathcal{F}).
    \end{equation*}
\end{lemma}
\begin{proof}
Let $f:X_1\to X_2$ be a continuous map, and let $\mathcal{F}_i,i=1,2$ be a sheaf of Abelian groups on $X_i$. A \emph{map} $(X_1,\mathcal{F}_1)\to (X_2, \mathcal{F}_2)$ over $f$ is a map of sheaves $f^*(\mathcal{F}_2)\to \mathcal{F}_1$ over $X_1$, or equivalently a map of sheaves $\mathcal{F}_2\to f_*(\mathcal{F}_1)$ over $X_2$.
Either version induces homomorphisms of the cohomology groups $H^\bullet(X_2;\mathcal{F}_2)\to H^\bullet(X_1; \mathcal{F}_1)$.

Let $p:X_1\times[0,1]\to X_1$ be the projection. We have two obvious maps $$(X_1,\mathcal{F}_1)\rightrightarrows (X_1\times [0,1],p^*(\mathcal{F}_1))$$ over the inclusions $x_1\mapsto (x_1,0), x_1\in X_1$ and $x_1\mapsto (x_1,1),x_1\in X_1$ respectively. Two maps $(X_1,\mathcal{F}_1)\to (X_2, \mathcal{F}_2)$ over $f$ are \emph{homotopic} if the underlying continuous maps $X_1\to X_2$ are homotopic via a homotopy $F:X_1\times [0,1]\to X_2$ such that there is a map $(X\times [0,1],p^*(\mathcal{F}_1))\to (X_2,\mathcal{F}_2)$ over $F$ which makes the next diagram commute:
$$
\begin{tikzcd}
(X_1,\mathcal{F}_1) \ar[r,shift left=.75ex]\ar[r,shift right=.75ex] & (X_1\times[0,1],p^*(\mathcal{F}_1))\ar[r] & (X_2,\mathcal{F}_2).
\end{tikzcd}
$$
If this is the case, the induced cohomology maps $H^\bullet(X_2;\mathcal{F}_2)\to H^\bullet(X_1;\mathcal{F}_1)$ are the same, see e.g.~\cite[Corollary II.11.8]{bredon}.

Now we can apply these remarks to the spaces $X\times Y$ and $X$, the sheaves $\mathcal{F}$ and $\pi^*\mathcal{F}$, and the maps $\pi$ and $i:X\to X\times Y$, the latter being is an embedding with a fixed $Y$-component.
\end{proof}
We are now ready to prove parts~\ref{item:ho_eq}-\ref{item:coho_iso} of Theorem \ref{chap3:main_thm_Spec_C}.
\begin{proof}[Proof of Theorem \ref{chap3:main_thm_Spec_C}, parts~\ref{item:ho_eq}-\ref{item:coho_iso}]
    The existence of an affine model and the first two properties have been proven in Theorem~\ref{chap2:main_thrm} and Corollary~\ref{chap3:contractible_fibers}. So it suffices to prove the third property, which we will do by checking it for the affine model $m_X:M_X\to X$ constructed in the proof of Theorem~\ref{chap2:main_thrm}.

    Let us first assume that the complex of sheaves is a single sheaf $\mathcal{F}$ in degree $0$. 
    We proceed by induction on the number $n$ of elements in an affine open cover $\{ U_{i}\}_{i=1}^{n}$ of $X$ used to construct an affine model in the proof of Theorem~\ref{chap2:main_thrm}.

    If $n = 2$, then the affine model $M_X$ for $X$ is the pushout of the following diagram
\[\begin{tikzcd}
	{U_{1}\cap U_{2}} & {U_{2}\times \mathbb{A}^n_\CC} \\
	{U_{1}\times \mathbb{A}^n_\CC}.
	\arrow[hook, from=1-1, to=1-2]
	\arrow[hook, from=1-1, to=2-1]
\end{tikzcd}\]
We showed in Lemma \ref{chap3:simple_projection_induce_iso} that the projections $U_{i}\times\mathbb{A}^n_\CC\to U_{i}$ induce isomorphisms on cohomology. Therefore, applying the open and closed versions of the Mayer-Vietoris exact sequence (see Theorem~\ref{appendix:Mayer_Vietoris_main_thrm}) and the $5$-lemma we deduce that in this case the natural homomorphism
\begin{equation}\label{eq:iso_sheaf_coho}
m^{*}_X: H^{*}(X; \mathcal{F}) \to H^{*}(M_X; m^{*}_X\mathcal{F}) 
\end{equation}
is an isomorphism.

    Now we describe the induction step. Assume that~(\ref{eq:iso_sheaf_coho}) is an isomorphism for all separated schemes $X$ of finite type over $\CC$  that have affine open covers with at most $n-1,n\geq 3$ elements. Suppose $X$ is a separated complex algebraic scheme of finite type that has an affine open cover $X=\bigcup_{i=1}^n U_i$ which we use to build an affine model $M_X\to X$. Recall from the proof of Theorem~\ref{chap2:main_thrm} that $M_X$ is the pushout of the following diagram:
\[\begin{tikzcd}
	{M_{1}} & {M_{2}\times \mathbb{A}^n_\CC} \\
	{U_{n}\times \mathbb{A}^n_\CC}
	\arrow[hook, from=1-1, to=1-2]
	\arrow[hook, from=1-1, to=2-1]
\end{tikzcd}\]
where $M_{1}$ is an affine model for $(U_{1}\cup \dots \cup U_{n-1})\cap U_{n}$, and $M_{2}$ is an affine model for $U_{1}\cup\dots\cup U_{n-1}$.

The induction hypothesis implies that each of the morphisms 
\[
M_{1}\to (U_{1}\cup\cdots \cup U_{n-1})\cap U_{n}, M_{2}\times\mathbb{A}^n_\CC \to U_{1}\cup\cdots \cup U_{n-1}
\]
induces a cohomology isomorphism for every sheaf on $(U_{1}\cup\cdots \cup U_{n-1})\cap U_{n}$, respectively on $U_{1}\cup\cdots \cup U_{n-1}$.

Applying the $5$-lemma and the Mayer-Vietoris sequence we see that~(\ref{eq:iso_sheaf_coho}) is an isomorphism for every sheaf $\mathcal{F}$ on $X$.

This completes the proof in the case when $\mathcal{F}^\bullet$ is a single sheaf in degree $0$. The general case follows from the spectral sequence $({}'E^{pq}_r,d'_r)$ described e.g.\ in~\cite[Section II.4.5]{Go58}.
\end{proof}

\subsubsection{Proof of Theorem~\ref{chap3:main_thm_Spec_C}, part~\ref{item:base_change}}
Let us first recall the definition of the base change map. Suppose we have a commutative (not necessarily Cartesian) diagram of continuous maps
$$
\begin{tikzcd}
W\ar[d,"p"]\ar[r,"h"] & X\ar[d,"f"]\\
Z\ar[r,"g"] & Y.
\end{tikzcd}
$$
Let $\mathcal{F}^\bullet$ be a complex of sheaves on $X$. Using the unit of the adjunction $(h^*,Rh_*)$ we get a map
$$Rf_* \mathcal{F}^\bullet\to Rf_* Rh_* h^*\mathcal{F}^\bullet=Rg_*Rp_*h^*\mathcal{F}^\bullet.$$
The adjunction $(g^*,Rg_*)$ gives us then the \emph{base change map}
$$\alpha:g^*Rf_* \mathcal{F}^\bullet\to Rp_*h^*\mathcal{F}^\bullet.$$

\begin{definition}
    An open continuous map $g: Z \to Y$ is called an \textit{weak acyclic submersion} if the following properties are satisfied:
    \begin{enumerate}[label=\Alph*]

        \item\label{itm:weak_acyclic_prop_a} For any $z\in Z$ the colimit map
        \begin{equation*}
            \colim \limits_{U \ni z} H^{\bullet}(g(U); \mathcal{F}^{\bullet}) \to \colim \limits_{U\ni z}H^{\bullet}(U; g^{*}\mathcal{F}^{\bullet})
        \end{equation*}
        is an isomorphism for any complex of sheaves $\mathcal{F}^{\bullet}$ on $Y$.
        \item\label{itm:weak_acyclic_prop_b} 
        For any complex of sheaves $\mathcal{F}^{\bullet}$ on $Y$, the map
        $$H^*(Y,\mathcal{F}^{\bullet})\to H^*(Z,g^*\mathcal{F}^{\bullet})$$
        is an isomorphism.
    \end{enumerate}
\end{definition}
\begin{example}
    Recall that a \textit{topological submersion} $g: Z\to Y$ is a continuous map such that for each $z \in Z$ there is an open neighborhood $z\in U \subset Z$ with a homeomorphism $\varphi: U \to g(U) \times V$ making the following triangle commute:
    \[\begin{tikzcd}
    	U && {g(U)\times V} \\
    	& {g(U)}.
    	\arrow["\varphi", from=1-1, to=1-3]
    	\arrow["g"', from=1-1, to=2-2]
    	\arrow["{\mathrm{pr}_{1}}", from=1-3, to=2-2]
    \end{tikzcd}\]
    A topological submersion has property~\ref{itm:weak_acyclic_prop_a}.
\end{example}
Our first goal is to prove the following lemma:
\begin{lemma}\label{chap4:lemma_base_change_for_affine_model_global_version}
    Consider the following Cartesian square of continuous maps
        \[\begin{tikzcd}
        	{X\times_{Y}Z} & X \\
        	Z & Y
        	\arrow["h", from=1-1, to=1-2]
        	\arrow["p"', from=1-1, to=2-1]
        	\arrow["\ulcorner"{anchor=center, pos=0.125}, draw=none, from=1-1, to=2-2]
        	\arrow["f", from=1-2, to=2-2]
        	\arrow["g"', from=2-1, to=2-2]
        \end{tikzcd}\]
    where $g$ is a weak acyclic submersion. For any complex of sheaves $\mathcal{F}^{\bullet}$ over $X$, the base change morphism
    \begin{equation*}
        \alpha_{X,\mathcal{F}^\bullet}:g^{*}Rf_{*}\mathcal{F}^{\bullet}\to Rp_{*}h^{*}\mathcal{F}^{\bullet}
    \end{equation*}
    is a quasi-isomorphism.
\end{lemma}
\begin{proof}
Let $i: U \hookrightarrow Z$ be an arbitrary open subset. Denote $j: V = f^{-1}(g(U)) \hookrightarrow X$ and consider the following Cartesian square
\begin{equation}\label{diagram_for_open_subsets}\begin{tikzcd}
	{V\times_{Y}U} & V \\
	U & {g(U)}.
	\arrow["{h_{U}}", from=1-1, to=1-2]
	\arrow["{p_{U}}"', from=1-1, to=2-1]
	\arrow["\ulcorner"{anchor=center, pos=0.125}, draw=none, from=1-1, to=2-2]
	\arrow["{f_{U}}", from=1-2, to=2-2]
	\arrow["{g_{U}}"', from=2-1, to=2-2]
\end{tikzcd}
\end{equation}
Set $\mathcal{F}^\bullet_U=j^*\mathcal{F}^\bullet$. By the definition of the base change map
$\alpha_{U, \mathcal{F}^{\bullet}_U}:g_{U}^{*}R(f_{U})_{*}\mathcal{F}^{\bullet}_U \to R(p_{U})_{*}h_{U}^{*}\mathcal{F}^{\bullet}_U$, the next diagram commutes:
\begin{equation}\label{diag:base_change_prelim}
\begin{tikzcd}
	{g_{U}^{*}R(f_{U})_{*}\mathcal{F}_{U}^{\bullet}} && {R(p_{U})_{*}h_{U}^{*}\mathcal{F}_{U}^{\bullet}} \\
	& {g_{U}^{*}R(f_{U})_{*}\mathcal{F}_{U}^{\bullet}}.
	\arrow["{\alpha_{U, \mathcal{F}_{U}^{\bullet}}}", from=1-1, to=1-3]
	\arrow["{\mathrm{id}}", from=2-2, to=1-1]
	\arrow["\Phi = \Phi_{U, \mathcal{F}_{U}^{\bullet}}"', from=2-2, to=1-3]
\end{tikzcd}
\end{equation}
Here $\Phi$ corresponds to the composition
\begin{equation}\label{base_change_def_phi}
    \varphi_{U, \mathcal{F}_{U}^{\bullet}}: R(f_{U})_{*}\mathcal{F}_{U}^\bullet \to R(f_{U})_{*}R(h_{U})_{*}h_{U}^{*}\mathcal{F}^{\bullet}_{U} = R(g_{U})_{*} R(p_{U})_{*}h_{U}^{*}\mathcal{F}_{U}^{\bullet}
\end{equation}
via the adjunction $(g_{U}^{*}, R(g_{U})_{*})$.

The adjunction $(g_{U}^{*}, R(g_{U})_{*})$ transforms~(\ref{diag:base_change_prelim}) into the commutative diagram
\[\begin{tikzcd}
	{R(g_{U})_{*}g_{U}^{*}R(f_{U})_{*}\mathcal{F}^{\bullet}_U} && {R(g_{U})_{*}R(p_{U})_{*}h_{U}^{*}\mathcal{F}^{\bullet}_U\cong R(f_{U})_{*} R(h_{U})_{*} h_{U}^{*}\mathcal{F}^{\bullet}_U} \\
	& {R(f_{U})_{*}\mathcal{F}^{\bullet}_U}.
	\arrow["{R(g_{U})_{*}\alpha_{U,\mathcal{F}^{\bullet}_U}}", from=1-1, to=1-3]
	\arrow["{\theta=\theta_{U, \mathcal{F}^{\bullet}_U}}", from=2-2, to=1-1]
	\arrow["{\varphi=\varphi_{U, \mathcal{F}^{\bullet}_U}}"', from=2-2, to=1-3]
\end{tikzcd}\]
Here $\theta$ is the unit of the adjunction $(g^{*}_{U}, R(g_{U})_{*})$ applied to the complex $Rf_{*}\mathcal{F}^{\bullet}_U$, and $\varphi$ is the map~(\ref{base_change_def_phi}).

Let now $z\in Z$ be a point. The $R\Gamma$ of $\theta$ coincides with the natural map
\begin{equation*}
    H^{*}(g(U);R(f_{U})_{*}\mathcal{F}^{\bullet}_U) \to H^{*}(U; g_{U}^{*}R(f_{U})_{*}\mathcal{F}^{\bullet}_U).
\end{equation*}
So the colimit of the maps $R\Gamma(\theta)$ over all neighborhoods $ U\ni z$ coincides with the map
\begin{equation*}
    \colim \limits_{U \ni z} H^{\bullet}(g_{U}(U); R(f_{U})_{*}\mathcal{F}^{\bullet}_U) \to \colim \limits_{U \ni z}H^{\bullet}(U; g_{U}^{*}R(f_{U})_{*}\mathcal{F}^{\bullet}_U).
\end{equation*}
Since $g$ is a weak acyclic submersion this map is in fact an isomorphism by property~\ref{itm:weak_acyclic_prop_a}.

Also note that the map $R\Gamma(\varphi)$ coincides with
\begin{equation*}
    H^{*}(V; \mathcal{F}^{\bullet}_U) \to  H^{*}(V\times_{Y}U; h_{U}^{*}\mathcal{F}^{\bullet}_U).
\end{equation*}
By property~\ref{itm:weak_acyclic_prop_b} applied to the complex $Rj_*\mathcal{F}_U$ the map $R\Gamma (\varphi)$ is an isomorphism, and hence so is the colimit of $R\Gamma (\varphi)$ over $U\ni z$.

After taking functorial resolutions we may assume that both $g^{*}Rf_{*}\mathcal{F}^{\bullet}$ and $ Rp_{*}h^{*}\mathcal{F}^{\bullet}$ are term-wise $\Gamma$-acyclic complexes of sheaves, and so are their restrictions to any open $U\subset X$. Then the above argument shows that $\alpha_{X,\mathcal{F}^\bullet}$ is a quasi-isomorphism on stalks, hence a quasi-isomorphism.
\end{proof}

\smallskip

Now we need to check that Lemma~\ref{chap4:lemma_base_change_for_affine_model_global_version} is applicable to the affine models constructed in Section~\ref{Construction_of_aff_model}.

\begin{lemma}\label{chap4:lemma_affine_model_is_almost_submersion}
    Let $Y$ be a separated scheme of finite type over $\CC$. The affine model $m_{Y}: M_{Y} \to Y$ constructed in the proof of Theorem~\ref{chap2:main_thrm} is a weak acyclic submersion.
\end{lemma}
\begin{proof}
    Property~\ref{itm:weak_acyclic_prop_b} follows from part~\ref{item:coho_iso} of Theorem~\ref{chap3:main_thm_Spec_C}. Openness and property~\ref{itm:weak_acyclic_prop_a} follow from the next observations.
    \begin{enumerate}
        \item 
        For a separated scheme $Z$ of finite type over $\CC$, the projection $Z \times \mathbb{A}^{n}_\CC \to Z$ is open and has property~\ref{itm:weak_acyclic_prop_a}.
        \item Let $g: Z \to Y$ be a morphism of separated schemes of finite type over $\CC$. If $Z$ has a closed cover $Z = Z_{1} \cup Z_{2}$ and $Y$ has an open cover $Y = Y_{1} \cup Y_{2}$ such that the restrictions $g|_{Z_{i}}: Z_{i} \to Y_{i}, i=1,2$ and $g|_{Z_{1}\cap Z_{2}}: Z_{1} \cap Z_{2} \to Y_{1} \cap Y_{2}$ are open and have property~\ref{itm:weak_acyclic_prop_a}, then the same is true for $g$. This follows from the Mayer-Vietoris Theorem~\ref{appendix:Mayer_Vietoris_main_thrm} and the fact that the direct limit is an exact functor.
        \item The proof of Theorem~\ref{chap2:main_thrm} is by induction on the number of elements in an affine open cover of a scheme $Y$. At each step the previous two remarks ensure that the morphism $m_Y:M_Y\to Y$ is open and has property~\ref{itm:weak_acyclic_prop_a}.
    \end{enumerate}
\end{proof}

\begin{proof}[Proof of Theorem~\ref{chap3:main_thm_Spec_C}, part~\ref{item:base_change}]
The statement follows from Lemmas~\ref{chap4:lemma_base_change_for_affine_model_global_version} and~\ref{chap4:lemma_affine_model_is_almost_submersion}.
\end{proof}

\subsection{Schemes over a field of positive characteristic}\label{example_etale_case}
Finally, in Section~\ref{example_etale_case} we will briefly discuss Example~\ref{ex:intro_char_p}. Let $\mathcal{C}$ be the category of separated schemes of finite type over $S=\Spec k$ with $k$ being an arbitrary field of characteristic $p$. Define a subclass of morphisms $\mathcal{W}$ in $\mathcal{C}$ as follows: $f: X\to Y$ belongs to $\mathcal{W}$ iff for any $l$-torsion sheaf $\mathcal{F}$ with $l \neq p$ the corresponding map of the \'etale cohomology groups
\[
 H^{*}_{\mbox{{\scriptsize \'et}}}(Y; \mathcal{F}) \to H^{*}_{\mbox{{\scriptsize \'et}}}(X; f^{*}\mathcal{F})
\]
is an isomorphism.
We will prove the following theorem:
\begin{theorem}\label{chap3:etale_case_admissibility}
    The class $\mathcal{W}$ defined above is adequate.
\end{theorem}
\begin{proof}
    It is clear that $\mathcal{W}$ is closed under composition. By e.g.\ \cite[Chapter VI, Corollary 4.20]{Mi80}, the projection $X\times_{\Spec k}\mathbb{A}^{1}_{\Spec k} \to X$ belongs to $\mathcal{W}$. By Theorem \ref{appendix:Mayer_Vietoris_main_thrm} there are Mayer-Vietoris \'etale cohomology exact sequences both for open and closed covers, and $f$ induces natural morphism between those. This implies the last condition.
\end{proof}
Further, we have the following analog of Theorem \ref{chap3:main_thm_Spec_C}
\begin{theorem}\label{chap3:main_thrm_etale_case}
    Let $X$ be a separated scheme of finite type over $\Spec k$, with $k$ being a field of characteristic $p$. Then $X$ admits an affine model $m_{X}: M_{X} \to X$ with the following properties: 
    \begin{enumerate}
    \item For any bounded below complex of $l$-torsion sheaves $\mathcal{F}^{\bullet}$ with $l \neq p$, the morphism $m_{X}$ induces an isomorphism
    \[
    H^{*}_{\mbox{\scriptsize\rm \'et}}(X; \mathcal{F}^{\bullet}) \to H^{*}_{\mbox{\scriptsize\rm \'et}}(M_{X}; m_{X}^{*}\mathcal{F}^{\bullet}).
    \]
    \item If $X'\to X$ is an arbitrary morphism of separated schemes of finite type over $k$, then the base change map along $m_X:M_X\to X$ is an isomorphism for any complex of \'etale sheaves on~$X'$.
    \end{enumerate}
\end{theorem}
\begin{proof}
    The existence of an affine model follows immediately from Theorems~\ref{chap3:etale_case_admissibility} and~\ref{chap2:main_thrm}. To pass from single sheaves to complexes of sheaves we use the spectral sequence 
    \[
    E_{1}^{p,q} \cong H^{q}_{\mbox{\scriptsize \'et}}(X; \mathcal{F}^{p})\Rightarrow H^{p+q}_{\mbox{\scriptsize \'et}}(X; \mathcal{F}^{\bullet}).
    \]
The proof of the last part repeats verbatim the proof of Theorem~\ref{chap3:main_thm_Spec_C}, part~\ref{item:base_change}.
%
\end{proof}

\section{Motivic nature of the Leray spectral sequence}\label{Motivic_nature_of_Leray}
\subsection{Complex algebraic case}\label{motivic_nature_complex}
Starting from Section~\ref{subsection:main_result_c} we will use the conventions stated at the beginning of Section~\ref{example_spec_c}, namely the base scheme $S=\Spec\CC$, by a sheaf on a separated scheme $X$ of finite type over $\CC$ we will understand a sheaf on $X(\CC)$ in the complex analytic topology etc. This time however, in line with~\cite{Ar03} we will write $V_\CC$ instead of $\mathcal{C}$ to denote the category of separated schemes of finite type over $\CC$ (see Definition~\ref{def_vs}).

Let $f:X\to Y$ be a continuous map of topological spaces. Suppose $\mathcal{F}^\bullet$ is a complex of sheaves on $X$. The \emph{Leray spectral sequence} of $f$ and $\mathcal{F}^\bullet$, denoted $(\mathcal{L}^{p,q}_r(f,\mathcal{F}^\bullet))$, is the Grothendieck spectral sequence for the composition of functors $R(pt_X)_*=R(pt_Y)_*\circ Rf_*$ where $pt_X:X\to pt, pt_Y:Y\to pt$ are maps to a point. We have
$$\mathcal{L}_2^{p,q}(f,\mathcal{F}^\bullet)= H^p(Y;R^q \mathcal{F}^\bullet)\Rightarrow H^{p+q}(X;\mathcal{F}^\bullet).$$

Equivalently, the Leray spectral sequence is obtained as follows: the canonical filtration on $Rf_*\mathcal{F}^\bullet$ induces a spectral sequence $(E^r_{p,q})$ that converges to $H^\bullet(Y,Rf_*\mathcal{F}^\bullet)=H^\bullet(X,\mathcal{F}^\bullet)$. We have then
$$
E_r^{p,q}=\mathcal{L}_{r+1}^{2p+q,-p}(f,\mathcal{F}^\bullet),r\geq 1,
$$
see~\cite[Exemple 1.4.8]{Deligne_theorie_de_hodge_2} or~\cite[Section 1]{Ar03}.
\subsubsection{Motivic spectral sequences}
This section contains the proof that
    if $f: X \to Y$ is an arbitrary proper morphism of separated schemes of finite type over $\CC$, then the Leray spectral sequence is motivic starting from the second page $E_{2}^{p,q} \cong H^{p}(Y; R^{q}f_{*}\mathcal{F})$, see Theorem~\ref{chap4:main_thm} below. The section follows the same general plan as Arapura's paper~\cite{Ar03}, with modifications which are due to the fact that we use different affine models.

We will shortly explain what the word ``motivic'' means in this statement, but first we mention a corollary of Theorem~\ref{chap4:main_thm}:
\begin{corollary}\label{chap4:main_corollary}
    For $f: X \to Y$ an arbitrary proper morphism of separated schemes of finite type over $\CC$ and an open subscheme $j:U\to X$, the corresponding Leray spectral sequence
    \[
    E_{2}^{p,q} \cong H^{p}(Y; R^{q}f_{*}\mathbb{Q}) \Longrightarrow H^{p+q}(X,X-U; \mathbb{Q})
    \]
    has a mixed Hodge structure (MHS) starting from the second page. This MHS is compatible with the one on the cohomology groups $H^{*}(X,X-U; \mathbb{Q})$, and is functorial in $f$.
\end{corollary}

In this section we also prove the following result.

\begin{theorem}\label{chap4:leray_on_affine_come_from_filtration}
    Let $f: X \to Y$ be an arbitrary proper morphism of separated schemes of finite type over $\CC$, and let $\mathcal{F}^\bullet$ be a complex of sheaves on $X$ with constructible cohomology. 
    
    If $Y$ is affine, then the Leray spectral sequence for $f$ and $\mathcal{F}^\bullet$ on $X$ is induced by a filtration of $Y$ by closed subschemes. 
    
    If $Y$ is arbitrary, then $Y$ has an affine model $m_Y:M_Y\to Y$ such that the Leray spectral sequence of $f$ and $\mathcal{F}^\bullet$ is isomorphic to that of $M_Y\times_{Y} X\to M_Y$ and the pullback of $\mathcal{F}^\bullet$.
\end{theorem}

\smallskip

In order to state Theorem~\ref{chap4:main_thm} we need a few definitions. 
\begin{definition}\label{def_vs}
Let $V_S$ be the category of separated schemes of finite type over a scheme $S$. If $S=\Spec k$ where $k$ is a field, we will abbreviate $V_S$ to $V_k$. Let $V_S^{2}$ be the category consisting of pairs $(X,Y)$ with $X\in V_S$ and $Y\subset X$ is a closed $S$-subscheme. A morphism $f: (X, Y)\to (X', Y')$ in this category is a morphism of schemes $f: X\to X'$ satisfying $f(Y) \subset Y'$. Analogously, one can define the category $V_S^{3}$ of triples $(X,Y,Z)$ such that $X,Y,Z\in V_S$ and $Z\subset Y, Y\subset X$ are closed subschemes.
\end{definition}
\begin{definition}
    A \emph{weak cohomology theory} is a collection of functors
    \[
    (\mathcal{H}^{i}: V^{2}_S \to \mathcal{A}),
    \]
    with $\mathcal{A}$ being a (fixed) Abelian category such as for any triple $(X,Y,Z)\in V^3_S$, the following sequence is exact:
    \begin{equation}\label{exact_seq_relative_coho_weak_theories}
    \cdots \to \mathcal{H}^{i}(X, Y)\to \mathcal{H}^{i}(X, Z)\to \mathcal{H}^{i}(Y,Z)\to\mathcal{H}^{i+1}(X,Y)\to\cdots.
    \end{equation}
\end{definition}
In this section we focus on the following two examples. For more examples see Section 1 of \cite{Ar03}.
\begin{example}\label{example:betti_obvious}
    $S=\Spec \CC, \mathcal{A} \cong Ab$, the category of Abelian groups, and $\mathcal{H} = H(-;\Z)$, the singular cohomology with integer coefficients.
\end{example}
\begin{example}\label{example:mhs_as_enhanced_betti}
    $S=\Spec \CC, \mathcal{A}$ is the category of mixed Hodge structures, and $\mathcal{H}=H^\bullet(-;\Q)$.
\end{example}
\begin{definition}
   A \emph{spectral sequence} $E_{r}^{\bullet, \bullet}$ in an Abelian category $\mathcal{A}$ is a collection of bigraded objects for each $r \geq r_{0}$ with differentials of bidegree $(r, -r+1)$, together with a filtered graded object $(H^{\bullet}, F^{\bullet})$, and isomorphisms
   \[
   E_{r+1}^{\bullet, \bullet} \cong H^{\bullet}(E_{r}^{\bullet, \bullet}),\ E^{\bullet, \bullet}_{\infty} \cong \operatorname{Gr}_{F}H^\bullet.
   \]
\end{definition}

The Leray spectral sequence can be viewed as a functor from the category of arrows in $V_\CC$ to the category of spectral sequences in $\mathcal{A}=Ab$ starting at page $r_0=2$.

\begin{definition}\label{chap4:def_of_lift_of_spec_seq}
    Consider a functor ${}_\mathcal{B}E$ from some category $\mathcal{C}$ to spectral sequences in an Abelian category $\mathcal{B}$, and let $\Phi: \mathcal{A} \to \mathcal{B}$ be an exact conservative functor between Abelian categories. Then a \emph{lift} of ${}_\mathcal{B}E$ to $\mathcal{A}$ is a functor ${}_\mathcal{A}E$ from $\mathcal{C}$ to spectral sequences in $\mathcal{A}$, and an isomorphism ${}_\mathcal{B}E \cong \Phi\circ {}_\mathcal{A} E$. 
\end{definition}
\begin{definition}
    An \emph{enhanced Betti theory} is a pair $((\mathcal{H}^i: V^{2}_S \to \mathcal{A}), \Phi)$ where $(\mathcal{H}^i)$ is a weak cohomology theory and $\Phi: \mathcal{A} \to Ab$ is an exact conservative functor. 
\end{definition}


\begin{definition}
    Let $((\mathcal{H}^i: V^{2}_S \to \mathcal{A}), \Phi)$ be an enhanced Betti theory, and let $\mathcal{C}$ be a category equipped with a functor $F:\mathcal{C}\to V^2_S$. Suppose ${}_{Ab}E$ is a functor from $\mathcal{C}$ to spectral sequences of Abelian groups such that the $H^\bullet$-component of ${}_{Ab}E$ is isomorphic to $\Phi\big(\mathcal{H}^\bullet(F(-))\big)$. Informally, we require that ${}_{Ab}E(c),c\in\mathcal{C}$ should converge to $\Phi\big(\mathcal{H}^\bullet(F(c))\big)$, functorially in $c$. The functor ${}_{Ab}E$ is \emph{motivic} if it has a lift ${}_\mathcal{A}E$ into $\mathcal{A}$ such that the $H^\bullet$-component of ${}_\mathcal{A}E$ is isomorphic to $\mathcal{H}^\bullet(F(-))$.
\end{definition}

\begin{example}\label{example:filtration_motivic_mhs}
Let us give a simple example of a motivic spectral sequence. Take $S=\Spec\CC$. Let $\mathcal{C}$ be the category of schemes of finite type over $\CC$ equipped with an exhausive filtration by closed subschemes. Let $F$ be the composition of the functor $\mathcal{C}\to V_\CC$ that forgets the filtration and the embedding $V_\CC\to V^2_\CC$ that takes $X\in V_\CC$ to $(X,\varnothing)$. Take $\mathcal{H}^\bullet$ to be the rational cohomology viewed as a mixed Hodge structure (see Example~\ref{example:mhs_as_enhanced_betti}). Then the spectral sequence induced by the filtration is motivic (\cite[Lemma 3.8]{Ar03}). 

The main idea is to construct an exact couple from the long exact sequences of the triples $(X,X_{a+1},X_a)$ where $X_a$ is the $a$-th term of the filtration. See Proposition~\ref{chap4:spec_seq_from_filtr_is_motivic} for a slightly more general version of the statement.
\end{example}

\subsubsection{Main result}\label{subsection:main_result_c}
Let $G$ be a fixed non-zero Abelian group. For a space $X$ we let $G_X$ denote the constant sheaf on $X$ with stalk $G$. If $Y\subset X$ is a closed subspace and $\mathcal{F}^\bullet$ is a complex of sheaves, we define the \emph{relative cohomology groups} 
$$H^\bullet(X,Y;\mathcal{F}^\bullet)=H^\bullet(X,j_!\mathcal{F}^\bullet)$$ where $j:X-Y\to X$ is the embedding. We set $H^\bullet(X;G)=H^\bullet(X;G_X)$ and $H^\bullet(X,Y;G)=H^\bullet(X,Y;G_X)$.


The category $V^2_\CC$ may be viewed as the category of all couples $(X,j_!G_U)$ where $j:U\to X$ is an open embedding. Let $(X,Y,Z)\in V_\CC^3$. Recall that we have an exact sequence of sheaves
$$0\to j_!^Z G_{X-Z}\to j_!^Y G_{X-Y}\to i^Y_*(j^{Z,Y}_! G_{Y-Z})\to 0$$
where $j^Z:X-Z\to X, j^Y:X-Y\to X, j^{Z,Y}:Y-Z\to Y$, and $i^Y:Y\to X$ are the embeddings. The resulting cohomology long exact sequence reads
\begin{equation}\label{exact_seq_relative_coho_sheaves}
\cdots\to H^i(X,Z;G)\to H^i(X,Y;G)\to H^i(Y,Z;G)\to H^{i+1}(X,Z;G)\to\cdots .
\end{equation}

\begin{definition}\label{def:classical_betti_theory}
An enhanced Betti theory $((\mathcal{H}^i: V^{2}_\CC \to \mathcal{A}), \Phi)$ is \emph{classical} if there is an Abelian group $G$ and an isomorphism $$\Phi(\mathcal{H}^\bullet(X,Y))\cong H^\bullet(X,Y;G)$$ of functors from $V_\CC^2$ to graded Abelian groups that moreover identifies the exact sequence~(\ref{exact_seq_relative_coho_weak_theories}) with the exact sequence~(\ref{exact_seq_relative_coho_sheaves}).
\end{definition}

The enhanced Betti theories from Examples~\ref{example:betti_obvious} and~\ref{example:mhs_as_enhanced_betti} are classical.

\smallskip

Now we need a natural domain of definition for the Leray spectral sequence. Let $\mathcal{C}$ be the category of all couples $(f,Z)$ where $f:X\to Y$ is a morphism in $V_\CC$, and $Z\subset X$ is a closed subscheme. Equivalently, $\mathcal{C}$ is the category of all couples $(f,j_! G_U)$ where $f$ is as above and $j:U\to X$ is an open embedding. We can then set ${}_{Ab}E$ to be the Leray spectral sequence functor from $\mathcal{C}$ to spectral sequences of Abelian groups:
$${}_{Ab}E^{p,q}_r(f,j_! G_U)=\mathcal{L}_r^{p,q}(f,j_! G_U),r\geq 2.$$
Finally, let the $F:\mathcal{C}\to V_\CC^2$ be the functor that takes $(f:X\to Y,Z)$ to $(X,Z)$.

\begin{theorem}\label{chap4:main_thm}
Let $\mathcal{H}^\bullet=((\mathcal{H}^i: V^{2}_\CC \to \mathcal{A}), \Phi)$ be a classical enhanced Betti theory. Set $\mathcal{C}^{prop}$ to be the full subcategory of $\mathcal{C}$ on all $(f,Z)$ such that $f$ is proper. Then the Leray spectral sequence functor is motivic on $\mathcal{C}^{prop}$ (with respect to $\mathcal{H}^\bullet$ and $F$).
\end{theorem}
The idea of the proof is as follows. Suppose we are given a morphism $f:X\to Y$. We can reduce the theorem to the case when $Y$ is affine by taking an affine model $m_{Y}: M_{Y} \to Y$ for $Y$ and pulling $f$ back to $M_Y$. If $Y$ is affine, we have a filtration $Y_{\bullet}$ of $Y$ by closed subsets and the pullback filtration $X_{\bullet} = f^{-1}(Y_{\bullet})$. The latter gives one an exact couple such that the resulting spectral sequence is motivic and isomorphic to the Leray spectral sequence of $f$ from the second page on.

\subsubsection{Filtered complexes of sheaves}

Recall that given a topological space $Z$, assumed locally compact and Hausdorff, and a complex of sheaves $\mathcal{F}^{\bullet}$ on $Z$ equipped with a bounded decreasing filtration $S$ by subcomplexes of sheaves, one has a spectral sequence
\[ 
E_{1}^{p,q} = H^{p+q}(Z;\operatorname{Gr}_{S}^p(\mathcal{F}^{\bullet})) \Longrightarrow  H^{p+q}(Z;\mathcal{F}^{\bullet}).
\]

One can construct a filtration $S$ on any complex of sheaves $\mathcal{F}^\bullet$ starting from an increasing filtration $Z_{\bullet}$ of $Z$ by closed subsets, which we assume exhaustive and finite:
\[
\varnothing = Z_{-1}\subset Z_{0}\subset\dots\subset Z_{M} = Z.
\]
One proceeds as follows. Let $k_{a}: Z- Z_{a} \hookrightarrow Z$ be the inclusion. One can define a subcomplex $S^{a}(Z_{\bullet}, \mathcal{F}^{\bullet})$ of $\mathcal{F}^{\bullet}$ degree-wise by the formula
\[
(k_{a})_{!}(k_{a})^{*}\mathcal{F}^{\bullet}.
\]
In other words, the subcomplex $S^{a}(Z_{\bullet}, \mathcal{F}^{\bullet})\subset\mathcal{F}^\bullet$ is obtained by restricting $\mathcal{F}^{\bullet}$ to the open subset $Z- Z_{a}$ and then extending by zero. This gives us a decreasing filtration on $\mathcal{F}^\bullet$.

Let $i_{a}: Z_{a} \hookrightarrow Z$ be the closed inclusion. One can check that the quotient complex
\[
    S^{a}/S^{a+1} \cong (i_{a})_{*}\mathcal{F}_{a}^{\bullet},
\]
where 
\[
 \mathcal{F}_{a} = (j_{a})_{!}(j_{a})^{*}(i_a)^*\mathcal{F}^{\bullet},
\]
and $j_{a}$ is the inclusion $j_{a}: Z_{a}- Z_{a-1}\hookrightarrow Z_{a}$. Therefore we obtain the following lemma.
\begin{lemma}\label{chap4:filtrartion_to_spec_seq}
    Let $Z_{\bullet}$ be an increasing finite exhaustive filtration of a locally compact Hausdorff topological space $Z$ by closed subsets. Then for each complex of sheaves $\mathcal{F}^{\bullet}$ on $Z$, there is the following spectral sequence:
    \[
    E_{1}^{p,q} \cong H^{p+q}(Z; (i_{p})_{*}\mathcal{F}^{\bullet}_{p}) \Longrightarrow H^{p+q}(Z; \mathcal{F}^{\bullet}).
    \]
\end{lemma}

$\square$

We will denote this spectral sequence $E^{p,q}_r(Z_\bullet,\mathcal{F}^\bullet)$.

The following lemma enables one to construct a convenient replacement for $(\mathcal{F}, S^{\bullet}(Z_{\bullet}, \mathcal{F}))$. We keep the notation and assumptions of Lemma~\ref{chap4:filtrartion_to_spec_seq}.
\begin{lemma}\label{chap4:inj_resolution_for_sheaf_in_filtered_derived_cat}
   Let $\mathcal{F}^\bullet\to \mathcal{I}^{\bullet}$ be a quasi-isomorphism of complexes of sheaves on $Z$, and let $Z_{\bullet}$ be a filtration of $Z$ with closed subsets. Assume that both $\mathcal{F}^\bullet$ and $\mathcal{I}^\bullet$ are bounded below, or that the cohomological dimension of $Z$ is finite. Then the natural morphism
  \[
  (\mathcal{F}^\bullet,   S^{\bullet}(Z_{\bullet},\mathcal{F})) \to ( \mathcal{I}^{\bullet}, S^{\bullet}( Z_{\bullet}, \mathcal{I}^{\bullet}))
  \]
  induces an isomorphism of the cohomology groups of the associated graded complexes of sheaves.
\end{lemma}
This is a variation of \cite[Lemma 3.3]{Ar03}. Note that in the proof of this lemma it is claimed that extension by zero along an open embedding preserves injective sheaves. This is not always the case: Let $Z$ be compact Hausdorff and take $j:U\to Z$ to be an inclusion with $U$ open and dense in $Z$. Let $\mathcal{F}$ be an injective sheaf on $U$ that has a nowhere vanishing section. Then this section does not extend to a section of $j_! \mathcal{F}$, so $j_!\mathcal{F}$ is not flabby, hence it is not injective.
\begin{proof}
    Since the functor $\operatorname{Gr}^{a}_{S} = (i_{a}\circ j_{a})_{!}\circ (i_{a}\circ j_{a})^{*}$ is exact, we see that the induced morphism $\operatorname{Gr}^{a}_{S}\mathcal{F}^\bullet \to \operatorname{Gr}^{a}_{S}\mathcal{I}^{\bullet}$ is a quasi-isomorphism as well. Applying the spectral sequence that calculates the hypercohomology using the canonical filtration (see e.g.~\cite[Theorem II.4.6.1]{Go58}), we obtain that the corresponding map $\operatorname{Gr}^{a}_{S}\mathcal{F}^\bullet \to \operatorname{Gr}^{a}_{S}\mathcal{I}^{\bullet}$ is an $H$-acyclic morphism.
\end{proof}

In the sequel we will prove that, given a morphism $f:X\to Y$ in $V_\CC$ with $Y$ affine, there is an exhaustive finite filtration of $Y$ such that the Leray spectral sequence of $f$ is isomorphic to the spectral sequence from Lemma \ref{chap4:filtrartion_to_spec_seq}. This will establish the motivic nature of the Leray spectral sequence due to the following proposition, c.f.\ Example~\ref{example:filtration_motivic_mhs}:
\begin{proposition}\label{chap4:spec_seq_from_filtr_is_motivic}
Let $((\mathcal{H}^i: V^{2}_\CC \to \mathcal{A}), \Phi)$ be a classical enhanced Betti theory with underlying Abelian group $G$. Suppose $X_{\bullet}$ is an exhaustive filtration of a separated scheme $X$ of finite type over $\CC$ by closed subschemes, and let $Z\subset X$ be another closed subscheme with $j:X-Z\to X$ being the embedding. Then the spectral sequence $E_{1}^{p,q}(X_{\bullet}, j_!G_{X-Z})\Rightarrow H^{p+q}(X,Z;G)$ is motivic.
\end{proposition}
\begin{proof}
    The proof is essentially the same as the proof of~\cite[Lemma 3.8]{Ar03}. 
\end{proof}

\smallskip

The key point of the proof that the Leray spectral sequence is motivic is the following definition.
\begin{definition}
   An increasing exhaustive filtration $Y_{\bullet}$ of $Y\in V_\CC$ by closed subschemes is called \emph{cellular} with respect to a sheaf $\mathcal{F}$ of Abelian groups on $Y$ if 
    \[
        H^{i}(Y; \mathcal{F}_{a}) = 0\ \text{unless}\ i=a.
    \]
\end{definition}

\begin{definition}\label{def:weakly_constr_sheaf}
A sheaf $\mathcal{F}$ of Abelian groups on $Y\in V_\CC$ is \emph{weakly constructible} if there is an exhaustive filtration $Y_\bullet$ of $Y$ by closed subschemes such that $\mathcal{F}$ is locally constant on every $Y_i-Y_{i-1}$.
\end{definition}
The following lemma due to Beilinson shows that cellular filtrations exist for affine schemes and weakly constructible sheaves:
\begin{lemma}[Beilinson]\label{chap4:Beilinson_lemma}
    If $\mathcal{F}$ is a weakly constructible sheaf on an affine $Y\in V_\CC$ of dimension $n$, then there exists a non-empty open subscheme $j: U\hookrightarrow X$ such that
    \[
    H^{i}(Y; j_{!}j^{*}\mathcal{F}) = 0\ \text{unless}\ i = n.
    \]
\end{lemma}
\begin{proof}
    See e.g.~\cite[Basic Lemma]{No18}.
\end{proof}
From this lemma, one can deduce the following corollary.
\begin{corollary}\label{chap4:cell_filtr_from_arbitrary_filtr}
    Let $Y$ be an affine $Y\in V_\CC$ of dimension $n$. Suppose that $\{ \mathcal{F}_{i}\}_{i=1}^{m}$ is a finite collection of weakly constructible sheaves on $Y$, and that
    \[
    Y_{0}'\subset Y_{1}'\subset\dots \subset Y_{n}' = Y
    \]
    is a chain of closed subschemes such that $\operatorname{dim}Y_{i}' \leq i$. Then there exists a filtration by closed subschemes
\[
Y_{0}\subset Y_{1}\subset \dots \subset Y_{n} = Y,
\]
such that 
\begin{itemize}
    \item $\operatorname{dim}Y_{i} =i$ for $i=0,\ldots, n$.
    \item $Y_{i}'\subset Y_{i}$ for $i=0,\ldots, n$.
    \item The filtration $Y_{\bullet}$ is cellular with respect to $\{\mathcal{F}_{i}\}_{i=1}^{m}$.
\end{itemize}
\end{corollary}

This is Lemma 3.7 in~\cite{Ar03}. The main idea is to apply Beilinson's lemma to $ \bigoplus\limits_{i=1}^{m}\mathcal{F}_{i}$.

Now let us review the basic properties of cellular filtrations. We have the following simple observation.
\begin{lemma}\label{chap4:first_page_of_spec_seq_with_cell_filtr}
    Let $Y_{\bullet}$ be a filtration of $Y\in V_\CC$ which is cellular with respect to a sheaf $\mathcal{F}$. Then
    \[
        H^{i}(Y; \mathcal{F}) \cong H^{i}(E_{1}^{\bullet, 0}(Y_{\bullet}, \mathcal{F}), d_{1}).
    \]
\end{lemma}
This is~\cite[Lemma 3.4]{Ar03}.
\begin{lemma}\label{chap4:functoriality_of_basic_spec_seq}
    Suppose $f: X \to Y$ is a proper morphism in $V_\CC$, let $Y_{\bullet}$ be an arbitrary filtration by close subschemes, and let $\mathcal{F}^\bullet$ be an arbitrary complex of sheaves over $X$. 
    Then setting $X_\bullet=f^{-1}(Y_\bullet)$ one has the following isomorphism of spectral sequences
    \[
    E_{1}^{p,q}(Y_{\bullet}, Rf_{*}\mathcal{F}^\bullet) \cong E_{1}^{p,q}(X_{\bullet}, \mathcal{F}^\bullet).
    \]
\end{lemma}
This is a variation of \cite[Lemma 3.9 and Corollary 3.10]{Ar03}. Note that the lemma is not true without assuming $f$ proper: one could take for instance $f$ to be an inclusion of $X=\mathbb{A}^1_\CC$ in $Y=\mathbb{P}^1_\CC$, the complex $\mathcal{F}^\bullet$ equal a constant sheaf on $X$, and the filtration of $Y$ equal $(Y- X)\subset Y$.
\begin{proof}
By Lemma \ref{chap4:inj_resolution_for_sheaf_in_filtered_derived_cat} we may replace $\mathcal{F}^\bullet$ with its Godement resolution.
    So we will prove the lemma if we show the following equivalence
    \[
    f_{*}S^{a}(X_{\bullet}, \mathcal{G}) \cong S^{a}(Y_{\bullet}, f_{*}\mathcal{G})
    \]
    for an arbitrary sheaf $\mathcal{G}$ on $X$.

    Let $k_{a}: X- X_{a} \hookrightarrow X$ and $K_{a}:Y- Y_{a}\hookrightarrow Y$ be the natural inclusions. Then $S^{a}(X_{\bullet}, \mathcal{G}) = (k_{a-1})_{!}k_{a-1}^{*}\mathcal{G}$, and similarly $S^{a}(Y_{\bullet}, f_{*}\mathcal{G}) = (K_{a-1})_{!}K_{a-1}^{*}f_{*}\mathcal{G}$. Note that since $f$ is proper, we have $f_{*} = f_{!}$. Using this and proper base change we get
    \begin{multline*}
        f_{*}S^{a}(X_{\bullet}, \mathcal{G}) = f_{*}(k_{a-1})_{!}k_{a-1}^{*}\mathcal{G} = (f\circ k_{a-1})_{!}k_{a-1}^{*}\mathcal{G} = (K_{a-1})_{!}f_{*}k_{a-1}^{*}\mathcal{G} = (K_{a-1})_{!}K_{a-1}^{*} f_{*}\mathcal{G} =\\= S^{a}(Y_{\bullet}, f_{*}\mathcal{G}).
    \end{multline*}
\end{proof}
\begin{lemma}
    Let $f:X\to Y$ be as in Lemma~\ref{chap4:functoriality_of_basic_spec_seq}, and suppose $\mathcal{F}^\bullet$ is a complex of sheaves on $X$. Let $Y_{\bullet}$ be a filtration of $Y$ by closed subschemes which is cellular with respect to the sheaf $\bigoplus\limits_{q\in\mathbb{N}} R^{q}f_{*}\mathcal{F}$. Then one has the following isomorphism
    \[
    E_{2}^{p,q}(X_{\bullet}, \mathcal{F}^\bullet) \cong H^{p}(Y; R^{q}f_{*}\mathcal{F}^\bullet).
    \]
\end{lemma}
This is~\cite[Corollaries 3.10 and 3.11]{Ar03}.

\smallskip

\begin{proposition}\label{chap4:leray_iso_to_spec_seq_from_filtr}
    Let $f:X\to Y$ be as in Lemma~\ref{chap4:functoriality_of_basic_spec_seq}, and suppose $\mathcal{F}^\bullet$ is a complex of sheaves on $X$. 
    Then, for every filtration $Y_{\bullet}$ of $Y$ by closed subschemes, there is a natural morphism of spectral sequences
    \[
        \mathcal{L}_{2}^{p,q}(f, \mathcal{F}^\bullet) \to E_{2}^{p,q}(X_{\bullet}, \mathcal{F}^\bullet)
    \]
    where $X_{\bullet} = f^{-1}(Y_{\bullet})$.

    Moreover, if the filtration $Y_{\bullet}$ is cellular with respect to each sheaf $R^q\mathcal{F}^\bullet$, this morphism is an isomorphism.
\end{proposition}
\begin{proof}
   See~\cite[Lemma 3.13]{Ar03} for the first part of this proposition, and ibid., proof of Theorem~3.1 for the second.
\end{proof}

\subsubsection{Applying affine models}

Summarizing, we have shown that if $f:X\to Y$ is a proper morphism in $V_\CC$, and $Y$ is affine, then there is a filtration $Y_{\bullet}$ of $Y$ such that the Leray spectral sequence $\mathcal{L}_{2}^{p,q}(f, \mathcal{F}^\bullet)$ is isomorphic to $E_{2}^{p,q}(X_{\bullet}, \mathcal{F}^\bullet)$. Here $\mathcal{F}^\bullet$ is a complex of sheaves on $X$ with constructible cohomology, and $X_{\bullet} = f^{-1}(Y_{\bullet})$. Affine models will enable us to reduce the case of an arbitrary target $Y$ to the case of affine $Y$.

Let $m_{Y}: M_{Y} \to Y$ be an affine model constructed in the proof of Theorem \ref{chap2:main_thrm}. Consider the following pullback square
\begin{equation}\label{chap4:pullback_square_to_reduce_to_aff_model}
\begin{tikzcd}
	{X\times_{Y}M_{Y}} & X \\
	M_Y & Y\arrow[ul, phantom, "\ulcorner", very near end].
	\arrow["{\operatorname{pr}_{1}}", from=1-1, to=1-2]
	\arrow["f'"', from=1-1, to=2-1]
	\arrow["f", from=1-2, to=2-2]
	\arrow["m_Y"', from=2-1, to=2-2]
\end{tikzcd}
\end{equation}
\begin{lemma}\label{chap4:reduce_to_aff}
    Assume $f$ proper, and let $\mathcal{F}^\bullet$ be a complex of sheaves on $X$ with constructible cohomology. One then has a natural isomorphism of the Leray spectral sequences
    \[
    \mathcal{L}_{r}^{p,q}(f,\mathcal{F}^\bullet) \cong \mathcal{L}_{r}^{p,q}(f',\mathcal{F}^\bullet),r\geq 2.
    \]
\end{lemma}

This is an analog of Lemma 3.14 in~\cite{Ar03}. We will prove a slightly more general version which will come in useful in the sequel.

\begin{lemma}\label{chap4:comp_of_spec_sequences}
Let
$$
\begin{tikzcd}
A' \ar[r,"n"]\ar[d,"g'"'] & A\ar[d,"g"]\\
B' \ar[r,"m"] & B\arrow[ul, phantom, "\ulcorner", very near end]
\end{tikzcd}
$$
be a Cartesian square of continuous maps of topological spaces. Let $\mathcal{G}^\bullet$ be a a complex of sheaves on $A$. Assume for every sheaf $\mathcal{H}$ on $B$ the pullback map $H^\bullet(B;\mathcal{H})\to H^\bullet(B';m^*(\mathcal{H}))$ is an isomorphism, and the base change maps
$$m^*(R^q g_*\mathcal{G}^\bullet)\to R^q g'_*(n^*(\mathcal{G}^\bullet))$$ are isomorphisms too. Then the map of the Leray spectral sequences 
\begin{equation}\label{chap4:comp_of_spec_sequences_map}
\mathcal{L}^{p,q}_r(g,\mathcal{G}^\bullet)\to\mathcal{L}^{p,q}_r(g',n^*\mathcal{G}^\bullet)
\end{equation}
is an isomorphism for $r\geq 2$.
\end{lemma}

\begin{proof}[Proof of Lemma~\ref{chap4:comp_of_spec_sequences}]
For $r=2$ the comparison map~(\ref{chap4:comp_of_spec_sequences_map}) is
$$\mathcal{L}^{p,q}_2(g,\mathcal{G}^\bullet)=H^p(B;R^q g_*\mathcal{G}^\bullet)\to H^q(B';m^*R^q g_*\mathcal{G}^\bullet)\to H^q(B';R^q g'_*(n^*\mathcal{G}^\bullet))=\mathcal{L}^{p,q}_2(g',n^*\mathcal{G}^\bullet).$$
Here the first arrow is a pullback map, and the second one is induced by a base change map; both are isomorphisms by our assumptions.
\end{proof}

\begin{proof}[Proof of Lemma~\ref{chap4:reduce_to_aff}] We use Lemma~\ref{chap4:comp_of_spec_sequences}, which is applicable by part~\ref{item:coho_iso} of Theorem~\ref{chap3:main_thm_Spec_C} and the assumption that $f$ is proper.
\end{proof}  

\subsubsection{Checking that the lift is well defined}

Let $((\mathcal{H}^i: V^{2}_\CC \to \mathcal{A}), \Phi)$ be a classical enhanced Betti theory with underlying Abelian group $G$. By Proposition~\ref{chap4:spec_seq_from_filtr_is_motivic}, for every exhaustive filtration $X_\bullet$ of $X\in V_\CC$ by closed subschemes the spectral sequence $E^{p,q}_r(X_\bullet,j_! G_U)$ is motivic where $j:U\to X$ is an open embedding. Let now $f:X\to Y$ be a proper morphism. Corollary~\ref{chap4:cell_filtr_from_arbitrary_filtr} and Proposition~\ref{chap4:leray_iso_to_spec_seq_from_filtr} allow us then to lift the individual Leray spectral sequence $\mathcal{L}_{2}^{p,q}(f,j_! G_U)$ to $\mathcal{A}$. The following lemmas show that this lift does not depend on any choices made in the construction.

\begin{lemma}\label{chap4:lift_doesnt_depend_on_filtr}
    The lift of the Leray spectral sequence $\mathcal{L}_{2}^{p,q}(f,j_! G_U)$ to an enhanced Betti theory does not depend on the choice of a cellular filtration of $Y$.
\end{lemma}
\begin{proof}
    Let $\{ Y_{i}^{(1)}\}_{i=0}^{n}$ and $\{Y_{i}^{(2)}\}_{i=0}^{n}$ be filtrations which are both cellular with respect to the sheaf $\mathcal{G} = \bigoplus\limits_{q\in\mathbb{N}} R^{q}f_{*}j_! G_U$. Consider the filtration
    \[
     Y_{0}^{(1)}\cup Y_{0}^{(2)}\subset Y_{1}^{(1)}\cup Y_{1}^{(2)}\subset \cdots \subset Y_{n}^{(1)}\cup Y_{n}^{(2)} = Y.
    \]
    By Corollary \ref{chap4:cell_filtr_from_arbitrary_filtr} there is a cellular filtration $\{ Y_{i}\}_{i=0}^{n}$ such that $Y_{i}^{(1)}\cup Y_{i}^{(2)} \subset Y_{i},\ i = 0,\ldots, n$. So the lifts of the Leray spectral sequence corresponding to the filtrations $\{ Y_{i}^{(1)}\}_{i=0}^{n}$ and $\{Y_{i}^{(2)}\}_{i=0}^{n}$ are both isomorphic to the lift of the Leray spectral sequence associated with $\{ Y_{i}\}_{i=1}^{n}$ by Proposition~\ref{chap4:leray_iso_to_spec_seq_from_filtr}, which shows the independence of the lifts on the choice of a cellular filtration.
\end{proof}
\begin{lemma}\label{chap4:funct_of_lifts_for_affine_targets}
        Let
    \[\begin{tikzcd}
    	{X'} & X \\
    	{Y'} & Y
    	\arrow["h", from=1-1, to=1-2]
    	\arrow["{f'}"', from=1-1, to=2-1]
    	\arrow["f", from=1-2, to=2-2]
    	\arrow[from=2-1, to=2-2]
    \end{tikzcd}\]
    be a commutative diagram of morphisms in $V_\CC$ with both $Y, Y'$ affine and both $f,f'$ proper. Let $Z'\subset X', Z\subset X$ be closed subschemes such that $h(Z')\subset Z$, and let $j':X'-Z'\to X', j:X-Z\to X$ be the corresponding open embeddings. The map of the Leray spectral sequences $$\mathcal{L}^{p,q}_r(f,j_!G_{X-Z})\to \mathcal{L}^{p,q}_r(f',j'_!G_{X'-Z'})$$ lifts to a map of the lifts to $\mathcal{A}$.
\end{lemma}
\begin{proof}
       Indeed, both Leray spectral sequences are constructed via an exact sequence of triples associated with some cellular filtrations $\{Y_{i}\}_{i=0}^{n},\ \{Y'_{i}\}_{i=0}^{n'}$ of $Y$ and $Y'$ respectively with $\operatorname{dim}Y_{i} = \operatorname{dim}Y'_{i} = i$. Hence we will be done if we show that one can choose the cellular filtrations in such a way that $h(Y_{i}') \subset Y_{i},\ i = 0,\ldots, n'$. Consider the following chain of Zariski closed subsets of $Y$:
    \[
     \overline{h(Y_{0}')} \subset \overline{h(Y_{1}')}\subset\cdots \subset \overline{h(Y_{n'}')}.
    \]
   Applying Corollary \ref{chap4:cell_filtr_from_arbitrary_filtr} we may assume that this chain is included in a cellular filtration $\{\tilde Y\}_{i=0}^n$ of~$Y$. Lemma~\ref{chap4:lift_doesnt_depend_on_filtr} gives us the required morphism of the lifts. 
\end{proof}
\begin{lemma}\label{chap4:lift_doesnt_depend_on_aff_model}
    Let $f:X\to Y$ be a proper morphism in $V_\CC$, and let $\mathcal{F}=j_! G_U$ where $j:U\to X$ is an open subscheme. The lift of the Leray spectral sequence $\mathcal{L}^{p,q}_r(f,\mathcal{F})$ to $\mathcal{A}$ does not depend on a choice of an affine model of $Y$ constructed using an affine open covering as in the proof of Theorem \ref{chap2:main_thrm}.
\end{lemma}
\begin{proof}
    Let $m_{Y}^{(1)}:M_{Y}^{(1)}\to Y,\ m_{Y}^{(2)}: M_{Y}^{(2)}\to Y$ be affine models constructed using affine open coverings $\{U_{i}\}$ and $\{V_{j}\}$ of $Y$ respectively. Consider the following pullback square
    \begin{equation}\label{diag_two_affine_models}
    \begin{tikzcd}
    	Z & {M_{Y}^{(1)}} \\
    	{M_{Y}^{(2)}} & Y\arrow[ul, phantom, "\ulcorner", very near end].
    	\arrow["n_1", from=1-1, to=1-2]
    	\arrow["n_2"', from=1-1, to=2-1]
    	\arrow["{m_{Y}^{(1)}}", from=1-2, to=2-2]
    	\arrow["{m_{Y}^{(2)}}"', from=2-1, to=2-2]
    \end{tikzcd}
    \end{equation}
    Let us pull back $f$ along the composition $Z\to M_Y^{(2)}\to Y$. The result is
\begin{equation}\label{diagram:aux_comparison_spec_seq}
\begin{tikzcd}
X\times_Y Z\ar[r]\ar[d,"f_Z"] & X\times_Y M_Y^{(2)}\ar[r]\ar[d,"f_2"] & X\ar[d,"f"]\\
Z\ar[r,"n_2"] & M_Y^{(2)}\ar[r] \arrow[ul, phantom, "\ulcorner", very near end]& Y\arrow[ul, phantom, "\ulcorner", very near end].
\end{tikzcd}
\end{equation}
Let us finish the proof of Lemma~\ref{chap4:lift_doesnt_depend_on_aff_model} assuming the following
\begin{lemma}\label{aux_lemma} 
Equip all schemes in the top row of~(\ref{diagram:aux_comparison_spec_seq}) with the pullbacks of the sheaf $\mathcal{F}$. Then the map of the Leray spectral sequence of $f_{2}$ to the Leray spectral sequence of $f_{Z}$ is an isomorphism from the second page on. 
\end{lemma}

Let $M_Z\to Z$ be an affine model. Lemmas~\ref{aux_lemma} and~\ref{chap4:reduce_to_aff} show that after pulling $f$ back to each of theschemes $M_Y^{(1)},M_Y^{(2)},M_Z$ we get Leray spectral sequences which are all isomorphic from the second term on. Furthermore, Lemma~\ref{chap4:funct_of_lifts_for_affine_targets} allows us to compare the lifts of the Leray spectral sequence of $f$ constructed using $M_Z$ and $M_Y^{(1)}$, and similarly for $M_Y^{(2)}$.
\end{proof}

\begin{proof}[Proof of Lemma~\ref{aux_lemma}] By Lemma~\ref{chap4:comp_of_spec_sequences}, it suffices to show that for any sheaf $\mathcal{G}$ on $M_Y^{(2)}$ the pullback map 
$$H^\bullet(M_Y^{(2)};\mathcal{G})\to H^\bullet(Z;n_2^*\mathcal{G})$$ is an isomorphism. We will do this by induction by the number $j$ of elements in an affine covering $\{U_{i}\}_{i=1}^{j}$ of $Y$. The argument is very similar to the one we used to prove part~\ref{item:coho_iso} of Theorem~\ref{chap3:main_thm_Spec_C}.

Let $j= 2$. Then 
$Z$ is the pushout of the following diagram
\[\begin{tikzcd}
	{(m_{Y}^{(2)})^{-1}(U_{1})\cap (m_{Y}^{(2)})^{-1}(U_{2})} & {(m_{Y}^{(2)})^{-1}(U_{2})\times\mathbb{A}^n_\CC} \\
	{(m_{Y}^{(2)})^{-1}(U_{1})\times \mathbb{A}^n_\CC}.
	\arrow[from=1-1, to=1-2]
	\arrow[from=1-1, to=2-1]
\end{tikzcd}\]
Applying Theorem~\ref{appendix:Mayer_Vietoris_main_thrm} we get the result we are after.

Now suppose the lemma holds true for all $Y$ that admit affine open covers $\{U_{i}\}_{i=1}^{j}$ with $j\geq 2$ elements. Let us prove that the lemma then also holds for $Y$ that can be covered by $j+1$ affine open sets, $Y=\bigcup_{i=1}^{j+1} U_i$. Denote $V = \bigcup_{i=1}^{j} U_{i}$. Then $Z$ is the pushout of the following diagram:
\[\begin{tikzcd}
	{(m_{Y}^{(2)})^{-1}(V) \cap (m_{Y}^{(2)})^{-1}(U_{j+1})} & {(m_{Y}^{(2)})^{-1}(U_{j+1}) \times \mathbb{A}^n_\CC} \\
	{(m_{Y}^{(2)})^{-1}(V)\times \mathbb{A}^n_\CC}
	\arrow[hook, from=1-1, to=1-2]
	\arrow[hook, from=1-1, to=2-1]
\end{tikzcd}\]
Applying induction hypothesis and Theorem~\ref{appendix:Mayer_Vietoris_main_thrm} again we obtain that the morphism $n_2:Z \to M_{Y}^{(2)}$ induces an isomorphism of the cohomology groups of an arbitrary sheaf.
\end{proof}
The following lemma shows that the lifts are functorial.
\begin{lemma}\label{chap4:functoriality_of_lifts}
We keep the notation and assumptions of Lemma~\ref{chap4:funct_of_lifts_for_affine_targets} except we no longer assume $Y$ or $Y'$ affine. Then the conclusion of Lemma~\ref{chap4:funct_of_lifts_for_affine_targets} is still true: the map of the Leray spectral sequences $$\mathcal{L}^{p,q}_r(f,j_!G_{X-Z})\to \mathcal{L}^{p,q}_r(f',j'_!G_{X'-Z'})$$ lifts to a map of the lifts to $\mathcal{A}$.
\end{lemma}
\begin{proof}
To reduce the case of arbitrary $Y, Y'$ to the affine case we use Theorem \ref{chap2:main_thrm} to construct affine models $M_{Y}, M_{Y'}$ such that the following diagram commutes:
    \[\begin{tikzcd}
    	{M_{Y'}} & {M_{Y}} \\
    	{Y'} & Y.
    	\arrow[from=1-1, to=1-2]
    	\arrow[from=1-1, to=2-1]
    	\arrow[from=1-2, to=2-2]
    	\arrow[from=2-1, to=2-2]
    \end{tikzcd}\]
    This is possible due to the functoriality part of Theorem \ref{chap2:main_thrm}. Then one can replace the initial diagram 
\begin{equation}\label{chap4:functoriality_of_lifts_initial_diagram}
\begin{tikzcd}
	{X'} & X \\
	{Y'} & Y
	\arrow["h", from=1-1, to=1-2]
	\arrow["{f'}"', from=1-1, to=2-1]
	\arrow["f", from=1-2, to=2-2]
	\arrow[from=2-1, to=2-2]
\end{tikzcd}
\end{equation}
with the following one:
    \[\begin{tikzcd}
    	{X\times_{Y}M_{Y'}} & {X\times_{Y} M_{Y}} \\
    	{M_{Y'}} & {M_{Y}},
    	\arrow[from=1-1, to=1-2]
    	\arrow[from=1-1, to=2-1]
    	\arrow[from=1-2, to=2-2]
    	\arrow[from=2-1, to=2-2]
    \end{tikzcd}\]
    and the Leray spectral sequences of both columns will be isomorphic to those of the columns of the original diagram~(\ref{chap4:functoriality_of_lifts_initial_diagram}) by Lemma~\ref{chap4:reduce_to_aff}, allowing us to make the desired reduction.
\end{proof}

\subsubsection{Proofs of the main results}

\begin{proof}[Proof of Theorem~\ref{chap4:leray_on_affine_come_from_filtration}]
    The theorem follows from Propositions~\ref{chap4:spec_seq_from_filtr_is_motivic} and \ref{chap4:leray_iso_to_spec_seq_from_filtr}, and Lemma~\ref{chap4:reduce_to_aff}.
\end{proof}
\begin{proof}[Proof of Theorem \ref{chap4:main_thm}]
    The theorem follows from Theorem~\ref{chap4:leray_on_affine_come_from_filtration}, and Lemmas~\ref{chap4:lift_doesnt_depend_on_filtr},~\ref{chap4:lift_doesnt_depend_on_aff_model} and~\ref{chap4:functoriality_of_lifts}.
\end{proof}

\begin{proof}[Proof of Corollary \ref{chap4:main_corollary}]
The corollary follows from Theorem~\ref{chap4:main_thm}.

\end{proof}

\subsubsection{Cohomology with compact support}

Now we will consider the case of the compactly supported Leray spectral sequence. Let $X$ be a Hausdorff locally compact space and $Y\subset X$ a closed subspace. This case is similar to what we've seen above in Section~\ref{motivic_nature_complex}, so we will only sketch the proofs. We set $H^\bullet_c(X;G)=H^\bullet_c(X;G_X)$ and $H^\bullet_c(X,Y;G)=H^\bullet_c(X;j_!G_X)$ where $j:X-Y\to X$ is an open embedding. 

If $f:X\to Y$ is a continuous map, and $\mathcal{F}^\bullet$ is a complex of sheaves on $X$, we let $(\mathcal{L}^{p,q}_{c,r}(f,\mathcal{F}^\bullet))$ denote the corresponding Leray spectral sequence with compact support. We have
$$\mathcal{L}_{c,2}^{p,q}(f,\mathcal{F}^\bullet)= H^p_c(Y;R^q \mathcal{F}^\bullet)\Rightarrow H^{p+q}_c(X;\mathcal{F}^\bullet).$$

\begin{definition}\label{def:classical_betti_theory_compact_supp}
We say than an enhanced Betti theory $((\mathcal{H}^i: V^{2}_\CC \to \mathcal{A}), \Phi)$ is \emph{classical with compact support} if there is an isomorphism $$\Phi(\mathcal{H}^\bullet(X,Y))\cong H^\bullet_c(X,Y;G)$$ of functors from $V_\CC^2$ to graded Abelian groups that moreover identifies the exact sequence~(\ref{exact_seq_relative_coho_weak_theories}) with the compactly supported version of~(\ref{exact_seq_relative_coho_sheaves}).
\end{definition}

\begin{example}
A basic example is $\mathcal{H}^\bullet(X,Y)=H^\bullet_c(X,Y;\Q)$ viewed as a mixed Hodge structure, cf.\ Example~\ref{example:mhs_as_enhanced_betti}.
\end{example}

Let the category $\mathcal{C}$ and the functor $F:\mathcal{C}\to V^2_\CC$ be as in Section~\ref{subsection:main_result_c}. We set ${}_{Ab}E$ to be the Leray spectral sequence with compact support viewed as a functor from $\mathcal{C}$ to spectral sequences of Abelian groups:
$${}_{Ab}E^{p,q}_{r}(f,j_! G_U)=\mathcal{L}^{p,q}_{c,r}(f,j_! G_U),r\geq 2.
$$

\begin{theorem}\label{chap4:main_thm_comp_supp}
Let $\mathcal{H}^\bullet=((\mathcal{H}^i: V^{2}_\CC \to \mathcal{A}), \Phi)$ be a classical enhanced Betti theory with compact support. Then the Leray spectral sequence with compact support is motivic (with respect to $\mathcal{H}^\bullet$ and~$F$).
\end{theorem}

Note that unlike Theorem~\ref{chap4:main_thm}, Theorem~\ref{chap4:main_thm_comp_supp} applies to arbitrary morphisms, not just proper ones.

\begin{proof} Theorem~\ref{chap4:main_thm_comp_supp} follows from Theorem~\ref{chap4:main_thm}, the next proposition, and the fact that any two compactifications of $X\in V_\CC$ are dominated by a third.
\end{proof}

 \begin{proposition}
    Let $f: X \to Y$ be an arbitrary morphism in $V_\CC$. Suppose $\mathcal{F}^\bullet$ is a complex of sheaves on $X$. Let $\overline{X}, \overline{Y}$ be compactifications of $X$ and $Y$ respectively such that there is a morphism $\overline{f}: \overline{X}\to \overline{Y}$ that makes the following diagram commute
\[\begin{tikzcd}
	{\overline{X}} & {\overline{Y}} \\
	X & Y.
	\arrow["{\overline{f}}", from=1-1, to=1-2]
	\arrow["{j^{X}}", hook, from=2-1, to=1-1]
	\arrow["f"', from=2-1, to=2-2]
	\arrow["{j^{Y}}"', hook', from=2-2, to=1-2]
\end{tikzcd}\]
Then one has the following isomorphism of spectral sequences
\[
\mathcal{L}_{c,r}^{p,q}(f,\mathcal{F}^\bullet) \cong \mathcal{L}_{r}^{p,q}(\overline{f}, j^X_{!}\mathcal{F}^\bullet),r\geq 2.
\]
\end{proposition}
\begin{proof}
Let us recall the construction of the Leray spectral sequences. Let $\mathcal{F}^\bullet \to \mathcal{I}^{\bullet}$ be the Godement resolution, which is term-wise $f_{!}$-acyclic and $f_*$-acyclic. Up to renumbering, the Leray spectral sequence $\mathcal{L}_{c,r}^{p,q}(f,\mathcal{F}^\bullet)$ is the spectral sequence induced by taking the $R\Gamma_!$ of the canonical filtration on $f_{!}\mathcal{I}^{\bullet}$ (which represents $Rf_! \mathcal{F}^\bullet$), c.f.\ the beginning of Section~\ref{motivic_nature_complex}. 

Similarly, the functor $j^X_!$ takes the components of the Godement resolution to soft sheaves, and $\bar f_*$ preserves soft sheaves as $\bar f$ is proper. So up to the same renumbering, the Leray spectral sequence $\mathcal{L}_{r}^{p,q}(\overline{f},j^X_{!}\mathcal{F}^\bullet)$ is induced by taking the $R\Gamma$ of the canonical filtration on $\bar f_*(j^X_!\mathcal{I}^{\bullet})$ (which represents $R\bar f_*(j^X_!\mathcal{F}^\bullet)$). 

Since $\bar f$ is proper, we have $\bar f_!=\bar f_*$, and hence
$$j^Y_!(f_{!}\mathcal{I}^{\bullet})=\bar f_*(j^X_!\mathcal{I}^{\bullet}).$$ Moreover, the functor $j^Y_!$ is exact, so it commutes with taking the canonical filtration. Denoting the latter by $\tau_{\leq n}$ we get
$$\Gamma_!(\tau_{\leq n}f_!\mathcal{I}^\bullet)=\Gamma\circ j^Y_!(\tau_{\leq n}f_!\mathcal{I}^\bullet)=\Gamma(\tau_{\leq n}(j^Y_! f_!\mathcal{I}^\bullet))=\Gamma(\tau_{\leq n}\bar f_*(j^X_!\mathcal{I}^{\bullet})),$$
which identifies the two spectral sequences.
\end{proof}
As an analog of Corollary \ref{chap4:main_corollary} we obtain the following corollary.
\begin{corollary}\label{chap4:main_corollary_compact_support}
    For $f: X \to Y$ an arbitrary morphism in $V_\CC$ and an open subscheme $j:U\to X$, the corresponding Leray spectral sequence
    \[
    E_{c,2}^{p,q} \cong H^{p}_c(Y; R^{q}f_{*}j_!\mathbb{Q}_U) \Longrightarrow H^{p+q}_c(X,X-U; \mathbb{Q})=H^{p+q}_c(U; \mathbb{Q})
    \]
    has a mixed Hodge structure (MHS) starting from the second page. This MHS is compatible with the one on the cohomology groups $H^{*}_c(U; \mathbb{Q})$, and is functorial in $f$.
\end{corollary}

$\square$

\subsection{Etale case}
Here we take $S=\Spec k$ where $k$ is an algebraically closed field of characteristic $p>0$. In this subsection we briefly outline the \'etale analog of the results of Section~\ref{motivic_nature_complex}. Let $G$ be a finite Abelian group of order coprime with $p$. The definition of a classical enhanced Betti theory, possibly with compact support (Definitions~\ref{def:classical_betti_theory} and~\ref{def:classical_betti_theory_compact_supp}) extends to the \'etale case in a straightforward way by replacing the singular cohomology $H^\bullet(-,-;G)$ with the \'etale cohomology $H^\bullet_{\mbox{\scriptsize\rm \'et}}(-,-;G)$.

Let $\mathcal{C}$ be the category of all couples $(f,Z)$ where $f:X\to Y$ is a morphism in $V_k$, and $Z\subset X$ is a closed subvariety, or equivalently, the category of all couples $(f,j_! G_U)$ where $f$ is as above and $j:U\to X$ is an open embedding. On $\mathcal{C}$ we have the Leray spectral sequence functor and the Leray spectral sequence functor with compact support:
$$(f,j_! G_U)\mapsto (\mathcal{L}^{p,q}_{r}(f,j_! G_U)), (f,j_! G_U)\mapsto (\mathcal{L}^{p,q}_{c,r}(f,j_! G_U)).$$ As in Section~\ref{subsection:main_result_c} we define the functor $F:\mathcal{C}\to V_k^2$ by $F(f:X\to Y,Z)=(X,Z)$.

Finally, let $\mathcal{C}^{prop}\subset \mathcal{C}$ be the full subcategory on all $(f:X\to Y,Z)$ such that such that $f$ is proper.
\begin{theorem}\label{chap4:etale_case_motivic_nature}
   Let $\mathcal{H}^\bullet=((\mathcal{H}^i: V^{2}_k \to \mathcal{A}), \Phi)$ be a classical enhanced Betti theory. Then the Leray spectral sequence is motivic on $\mathcal{C}^{prop}$ with respect to $\mathcal{H}^\bullet$ and $F$.

   Let $\mathcal{H}^\bullet=((\mathcal{H}^i: V^{2}_k \to \mathcal{A}), \Phi)$ be a classical enhanced Betti theory with compact support. Then the Leray spectral sequence with compact support is motivic on $\mathcal{C}$ with respect to $\mathcal{H}^\bullet$ and $F$.
\end{theorem}
The proof of this theorem mirrors the proofs of Theorems~\ref{chap4:main_thm} and~\ref{chap4:main_thm_comp_supp}.

An \'etale sheaf is \emph{weakly constructible} if it is locally constant on every consecutive difference of some filtration by closed subschemes, cf.\ Definition~\ref{def:weakly_constr_sheaf}.

\begin{theorem}\label{chap4:etale_case_existence_of_filtration}
    Let $f:X\to Y$ be a proper morphism in $V_k$, and let $\mathcal{F}$ be a weakly constructible \'etale sheaf on $X$ of torsion Abelian groups of order coprime with $p$. Assume $Y$ has a smooth point. 
    If $Y$ is affine over $\Spec k$, then there is a filtration $Y_{\bullet}$ of $Y$ by closed subschemes such that the corresponding spectral sequence is isomorphic, starting from the second page, to the Leray spectral sequence of the morphism $f$.
    
    Moreover, there is an affine model $m_{Y}: M_{Y} \to$ such that the \'etale Leray spectral sequence of $f$ and $\mathcal{F}$ is isomorphic to the Leray spectral sequence of $M_{Y}\times_{Y} X \to M_{Y}$ and $m_{Y}^{*}(\mathcal{F})$.
\end{theorem}
The proof of this theorem is similar to the proof of the Theorem~\ref{chap4:leray_on_affine_come_from_filtration}.

%

\smallskip

The only difference from the complex algebraic case is the proof of Beilinson's lemma. In the \'etale case the statement follows from Lemma 3.3 in Beilinson's original paper~\cite{Be87}:
\begin{lemma}\label{chap4:original_Beilinson_lemma}
    Let $f: X \to Y$ be any morphism of separated schemes of finite type over $\Spec k$. Let $r_{l}: U_{l} \hookrightarrow X$ be a finite set of affine open embeddings and $M$ a perverse sheaf of middle perversity. Then there exists a perverse sheaf of middle perversity $N \twoheadrightarrow M$ such that
    \[
    H^{i}_{\mbox{{\scriptsize\rm \'et}}} (Y;\left(f|_{U_{l}}\right)_{*} N|_{U_{l}}) = 0\ \text{for}\ i < 0.
    \] 
    Moreover, if $X$ is quasi-projective, then $N$ can be chosen of the form $j_{!}(M|_{V})$, where $j: V \hookrightarrow X$ is a certain affine open embedding.
\end{lemma}
A perverse sheaf of middle perversity is a complex of sheaves with some additional properties. We will now state two of those, which imply an analog of Lemma \ref{chap4:Beilinson_lemma} in positive characteristic. Let $Y$ be an affine smooth variety of dimension $n$, and let $\mathcal{F}$ be a locally constant sheaf on $Y$.
\begin{itemize}
    \item The complex $\mathcal{F}[n]$ is a perverse sheaf of middle perversity. See e.g.~\cite[Section III.2]{Rein01}
    \item Let $j: Y \hookrightarrow Y'$ be an affine open embedding. Then $j_{!}\mathcal{F}$ is a perverse sheaf of middle perversity. This is~\cite[Corollaire 4.1.3]{BBD82}.
\end{itemize}
For our purposes we need the following corollary of Lemma \ref{chap4:original_Beilinson_lemma}:
\begin{corollary}\label{chap4:Beilinson_lemma_positive_char}
    Let $X$ be a separated affine scheme of finite type over $\Spec k$ of dimension $n$. Let $\mathcal{F}$ be a weakly constructible sheaf on $X$. Then there exists a Zariski open subset $j:U \hookrightarrow X$ such that 
    \[
    H^{k}_{\mbox{{\scriptsize\rm \'et}}}(X; j_{!}j^{*}\mathcal{F}) = 0\ \text{if}\ k\neq n.
    \]
\end{corollary}
\begin{proof}
    Consider the reduced scheme $i: X^{red} \hookrightarrow X$. By e.g.~\cite[Proposition 59.45.4 (3)]{Stacks} $i^{*}$ induces an isomorphism on \'etale cohomology. Therefore, we may assume that the scheme $X$ is reduced.
    
    Since base field $k$ is algebraically closed and $X$ is reduced, $X$ is in fact geometrically reduced. Then by e.g.~\cite[Lemma 33.25.7]{Stacks} there is a Zariski open affine subset $j_{V}:V\hookrightarrow X$ which is a smooth affine variety of dimension $n$. Since $\mathcal{F}$ is weakly constructible one may assume that, possibly after shrinking $V$, the restriction $j_{V}^{*}\mathcal{F}$ is locally constant. Then the complex $j_{V}^{*}\mathcal{F}[n]$ is a perverse sheaf of middle perversity on $V$. Since $V$ is affine, the complex $(j_{V})_{!}j_{V}^{*}\mathcal{F}[n]$ is a perverse sheaf of middle perversity on $X$.

    Applying Lemma \ref{chap4:original_Beilinson_lemma} with $Y = \Spec k$, $f$ equal the structure morphism, and $r_{l}$ being the identity map $X\to X$, one can deduce that there is an affine Zariski open subset $j_{W}: W\hookrightarrow X$ such that 
    \[
    H^{k}_{\mbox{{\scriptsize \'et}}}(X; (j_{W})_{!}j_{W}^{*}(j_{V})_{!}j_{V}^{*}\mathcal{F}[n]) = 0\ \text{if}\ k<0.
    \]
    Consider the intersection $j: U = W\cap V\hookrightarrow X$. By base change, we have $(j_{W})_{!}j_{W}^{*}(j_{V})_{!}j_{V}^{*}\mathcal{F}[n]=j_{!}j^{*}\mathcal{F}[n]$, so
    \[
    H^{k}_{\mbox{{\scriptsize \'et}}}(X; j_{!}j^{*}\mathcal{F}[n]) = 0\ \text{if}\ k < 0.
    \]
    Apriori the intersection $W\cap V$ may be empty. However, from the proof of Lemma \ref{chap4:original_Beilinson_lemma}, the open set $W$ can be chosen to be the complement of a generic hyperplane section. Hence one can choose $W$ such that the intersection $U = W\cap V$ is non empty.

    We have $H^{k}_{\mbox{{\scriptsize \'et}}}(X; j_{!}j^{*}\mathcal{F}) = 0$ for $k<n$.
    But since $X$ is affine, by e.g.~\cite[Theorem 15.1]{Mi13}, the cohomological dimension of $X$ is $n$, which implies $H^{k}_{\mbox{{\scriptsize \'et}}}(X; j_{!}j^{*}\mathcal{F}) = 0$ for $k > n$.
\end{proof}

\begin{appendices}

\section{Mayer-Vietoris sequence}\label{appendix_Mayer_Vietoris}
\setcounter{equation}{0}
\renewcommand{\theequation}{\thesection.\arabic{equation}}
In this appendix we will deduce the Mayer-Vietoris sequence from the \v{C}ech-to-cohomology spectral sequence. Let $\mathcal{C}$ be a small category which has all pullbacks, i.e.\ limits of diagrams of the type $\bullet\to \bullet\gets \bullet$. We will use the following notions of a coverage and a sheaf with respect to a coverage which can be found e.g.\ in \cite[C2.1]{Jo02}:
\begin{definition}
A \emph{coverage} on the category $\mathcal{C}$ consists of a function assigning to each object $X \in \mathcal{C}$ a collection of families of morphisms $\{ f_{i}: U_{i}\to X\}_{i\in I}$, called \emph{covering families}, such that
\begin{itemize}
    \item if $\{ f_{i}: U_{i}\to X\}$ is a covering family and $g: Y\to X$ is a morphism in $\mathcal{C}$, then the pullback family $\{g^{*}(f_i):g^{*}U_{i} \to Y\}$ is a covering family of $Y$.
\end{itemize}
\end{definition}

Let $\mathcal{S}$ be a coverage on $\mathcal{C}$, and let $X \in \mathcal{C}$ be an object. We denote the \emph{slice category} of arrows $Y\to X$ in $\mathcal{C}$ by $\mathcal{C}_{/X}$. The coverage $\mathcal{S}$ induces a coverage on $\mathcal{C}_{/X}$, denoted $\mathcal{S}_{/X}$. If $\mathcal{F}$ is a presheaf on $\mathcal{C}_{/X}$, and $Y\to X$ is an object of $\mathcal{C}_{/X}$, we will sometimes write $\mathcal{F}(Y)$ instead of $\mathcal{F}(Y\to X)$ by abuse of notation. 

\begin{definition}
        Let $\mathcal{C}'$ be a category with pullbacks. A presheaf $\mathcal{F}$ of Abelian groups on $\mathcal{C}'$ satisfies the \emph{sheaf axiom} for a family of morphisms $\{f_{i}: U_{i} \to Z\}_{i\in I}$ in $\mathcal{C}'$ if the following holds: suppose $s_{i} \in \mathcal{F}(U_i)$ are such that for all $j, k \in I$ we have $\mathcal{F}(g)(s_{j}) = \mathcal{F}(h)(s_{k})$ where $g: U_{j}\times_{Z} U_{k} \to U_{j}$ and $h: U_{j}\times_{Z} U_{k}$ are the natural projections; then there exists a unique $s \in \mathcal{F}(Z)$ such that $\mathcal{F}(f_{i})(s) = s_{i}$ for each $i\in I$.
        
        A \emph{sheaf of Abelian groups} on $\mathcal{C}_{/X}$ is a presheaf of Abelian groups on 
        $\mathcal{C}_{/X}$ satisfying the sheaf axiom for any covering family in the coverage $\mathcal{S}_{/X}$ on $\mathcal{C}_{/X}$. 
\end{definition}

\begin{remark}
    The definitions above are~\cite[C2.1, Definition 2.1.1]{Jo02} and~\cite[C2.1, Definition 2.1.2]{Jo02}; the definition of a sheaf simplifies a little in our case because $\mathcal{C}_{/X}$ has all pullbacks.
\end{remark}
\begin{remark}
    Every coverage generates a Grothendieck topology with the same category of sheaves of Abelian groups. Moreover, Grothendieck topologies are in a natural bijection with the \emph{saturated} coverages (see e.g.~\cite[Proposition 6.35]{Mi25}).
\end{remark}

Let $X\in\mathcal{C}$. Let $\mathop{Sh}_X(\mathcal{S})$, respectively $\mathop{PreSh}_X(\mathcal{S})$ denote the category of presheaves, respectively sheaves of Abelian groups on $\mathcal{C}_{/X}$. We will be omitting $\mathcal{S}$ from this notation when it is clear which $\mathcal{S}$ is meant. The category $\mathop{Sh}_X(\mathcal{S})$ is a reflective subcategory of $\mathop{PreSh}_X(\mathcal{S})$ (see e.g.~\cite[C2.1, Corollary 2.11.1]{Jo02}). Let $\mathop{Forget}: \mathop{Sh}_X(\mathcal{S}) \hookrightarrow \mathop{PreSh}_X(\mathcal{S})$ denote the forgetful functor, and let $\mathop{Sh}: \mathop{PreSh}_X(\mathcal{S}) \to \mathop{Sh}_X(\mathcal{S})$ be the left adjoint of $\mathop{Forget}$, the sheafification functor. Observe that $\mathop{Sh}$ is exact. By the construction in~\cite[Section 19.7]{Stacks} we see that $\mathop{Sh}_X(\mathcal{S})$ has enough injectives.

Next, observe that the global sections functor $\Gamma: \mathop{Sh}_X(\mathcal{S}) \to Ab$ 
can be written as the composition $\widehat{\Gamma}\circ \mathop{Forget}$ where $\widehat{\Gamma}: \mathop{PreSh}_X(\mathcal{S}) \to Ab$ is the global sections functor on the category of presheaves. The functor $\mathop{Forget}$ preserves kernels as it is a right adjoint, and so does $\widehat{\Gamma}$ because it is exact. So the composition $\Gamma$ also preserves kernels and is therefore left exact.

So we can define the \emph{$q$-th cohomology group} functor $H^q(X;-):\mathop{Sh}_{X}(\mathcal{S})\to Ab$ as $R^q\Gamma(-)$.

\smallskip

Let us define the \v{C}ech complex associated with a family of morphisms. Let $\mathcal{U} = \{U_{i}\to X\}_{i\in I}$ be an arbitrary set of morphisms with a fixed target $X$ in $\mathcal{C}$, and let $\mathcal{F}$ be an arbitrary presheaf of Abelian groups on $\mathcal{C}_{/X}$. Then we define the \emph{\v{C}ech complex} $C^{\bullet}(\mathcal{U}; \mathcal{F})$ as follows:
\[
    \prod\limits_{i_{0}\in I}\mathcal{F}(U_{i_{0}}) \to \prod\limits_{i_{0},i_{1}\in I}\mathcal{F}(U_{i_{0}}\times_{X} U_{i_{1}}) \to \prod\limits_{i_{0}, i_{1}, i_{2}\in I}\mathcal{F}(U_{i_{0}}\times_{X}U_{i_{1}}\times_{X} U_{i_{2}})\to \cdots ,
\]
see~\cite[Definition 59.18.1]{Stacks}. Set $\check{H}^{p}(\mathcal{U}; \mathcal{F}) = H^{p}(C^{\bullet}(\mathcal{U}; \mathcal{F}))$. 
\begin{remark}
    If $I$ is totally ordered, then we can also define the \emph{ordered \v{C}ech complex} $\mathcal{C}_{o}^{\bullet}(\mathcal{U}; \mathcal{F})$ as the following quotient complex of $\mathcal{C}^{\bullet}(\mathcal{U}; \mathcal{F})$:
    \begin{equation*}
        \prod\limits_{i_0 \in I} \mathcal{F}(U_{i_0}) \to \prod\limits_{i_0 < i_1 \in I} \mathcal{F}(U_{i_0}\times_{X} U_{i_1}) \to \prod\limits_{i_0 < i_1 < i_2\in I}\mathcal{F}(U_{i_0}\times_{X} U_{i_1}\times_X U_{i_2}) \to \cdots
    \end{equation*}
    The natural factorization map $\mathcal{C}^{\bullet}(\mathcal{U}; \mathcal{F}) \to \mathcal{C}^{\bullet}_o(\mathcal{U}; \mathcal{F})$ is a quasi-isomorphism.
\end{remark}
\begin{definition}
    A family $\mathcal{U} = \{ U_{i} \to X\}_{i\in I}$ of morphisms in $\mathcal{C}$ with the same target $X\in\mathcal{C}$ is called an \emph{admissible cover} if for any sheaf $\mathcal{F}$ on $\mathcal{C}_{/X}$ the following isomorphism of functors holds
    \begin{equation}\label{iso_func}
        H^{0}(U; \mathcal{F}) = \mathcal{F}(X\stackrel{\mathrm{id}}{\longrightarrow} X)\cong \check{H}^{0}(\mathcal{U}; \mathcal{F})\circ \mathop{Forget}(\mathcal{F}).
    \end{equation}
\end{definition}

Note that an element of an admissible cover may or may not belong to any of the covering families for $X$. We are now ready to state the main result. 
\begin{theorem}\label{thm_mv}
    Let $\mathcal{U}=\{U_{i}\to X\}_{i\in I}$ be an admissible cover, and let $\mathcal{F}$ be a sheaf on $\mathcal{C}_{/X}$. There is a spectral sequence
    \begin{equation}\label{Cech_to_cohomology_spec_sequence}
    E^{p,q}_{2} \cong \check{H}^{p}(\mathcal{U}; \mathcal{H}^{q}(\mathcal{F}))\Longrightarrow H^{p+q}(X; \mathcal{F}),
    \end{equation}
    where $\mathcal{H}^q(\mathcal{F})$ is the presheaf on $\mathcal{C}_{/X}$ that takes $(Y\to X) \in \mathcal{C}_{/X}$ to $\mathcal{H}^{q}(\mathcal{F})(Y) = H^{q}(Y; \mathcal{F})$.
\end{theorem}
\begin{proof}
    This is the Grothendieck spectral sequence associated with the following composition of functors
    \[
    Sh_X\xrightarrow{\mathop{Forget}(-)} PreSh_X \xrightarrow{\check{H}^{0}(\mathcal{U}; -)} Ab.
    \]

    Indeed, the right derived functor of the composition $\check{H}^{0}(\mathcal{U};-)\circ \mathop{Forget}(-) \cong H^{0}(X; -)$ is the derived functor of the global sections over $X$. So
    \[
        R^{p+q}(\check{H}^{0}(\mathcal{U};\mathcal{F})\circ \mathop{Forget}(\mathcal{F})) \cong H^{p+q}(X; \mathcal{F}).
    \]

    By e.g.~\cite[Theorem 59.18.8]{Stacks} the $p$-th right derived functor of $\check{H}^{0}(\mathcal{U};-)$ is $\check{H}^{p}(\mathcal{U}; -)$. It is clear that the $q$-th derived functor of $\mathop{Forget}: \mathop{Sh_X} \to \mathop{PreSh}_X$ takes a sheaf $\mathcal{F}$ to the presheaf $\mathcal{H}^{q}(\mathcal{F})$. Therefore, $\Bigl(R^{p}(\check{H}^{0}(\mathcal{U}; -)) \circ R^{q}(\mathop{Forget}(-))\Bigr)(\mathcal{F}) \cong \check{H}^{p}(\mathcal{U}; \mathcal{H}^{q}(\mathcal{F}))$.
\end{proof}

As an example, consider an admissible cover $\mathcal{U} = \{ \{U \to X\}, \{ V \to X\}\}$. In that example for simplicity we totally order $\mathcal{U}$ and use the ordered \v{C}ech complex.
Note that 
\[
E^{0,n}_{2} \cong\check{H}^{0}(\mathcal{U}; \mathcal{H}^{n}(\mathcal{F})) \cong \ker\left[d: C^{0}(\mathcal{U}; \mathcal{H}^{n}(\mathcal{F}))\to C^{1}(\mathcal{U}; \mathcal{H}^{n}(\mathcal{F}))\right].
\]
So $E^{0,n}_{2}$ is the subgroup of $H^{n}(U; \mathcal{F})\oplus H^{n}(V; \mathcal{F})$ which consists of elements $(t,s)$ such that $t|_{U\times_{X}V} = s|_{U\times_{X}V}$ in $H^{n}(U\times_{X}V; \mathcal{F})$.

In a similar manner, $E_{2}^{1,n-1}$ is the quotient group of $H^{n-1}(U\cap V; \mathcal{F})$ by the subgroup generated by the differences $t - s$ where $t$ is an element coming by restriction from $H^{n-1}(U; \mathcal{F})$, and similarly for $s$.

Observe that $E_{2}^{p,q} = 0$ for $p \geq 2$, so $E_{2}^{p,q}= E^{p,q}_{\infty}$.
Hence we have the following exact sequence for every $n \geq 0$:
\begin{equation}\label{exact_seq}
0\to E_{2}^{1,n-1}\to H^{n}(X; \mathcal{F}) \to E^{0,n}_{2}\to 0.
\end{equation}

With this description in mind, it is easy to glue the exact sequences~(\ref{exact_seq}) into one long exact sequence, called the \emph{Mayer-Vietoris sequence}.

Let us give a few examples of admissible covers.

\begin{example}\label{appendix_ex_of_adm_cover_site}
    All elements of $\mathcal{U}$ belong to $\mathcal{S}$. This follows immediately from the definition of a sheaf.
\end{example}
\begin{example}\label{appendix_ex_of_adm_cover_extenstion_of_sites}
    Let $\mathcal{S}' \subset \mathcal{S}$ be two coverages on $\mathcal{C}$. Let $\mathcal{U} = \{ U_{i}\to X\}_{i\in I}$ be an admissible cover with respect to the coverage $\mathcal{S}'$. Then $\mathcal{U}$ is also admissible with respect to the coverage $\mathcal{S}$.

    Indeed, let $\operatorname{Res}: Sh_{X}(\mathcal{S}) \to Sh_{X}(\mathcal{S}')$ be the restriction functor. By definition,  for any sheaf $\mathcal{F} \in Sh_{X}(\mathcal{S})$, the functor $\operatorname{Res}$ acts as the identity on the Abelian group $\mathcal{F}(Y\to X)$ for $(Y\to X)  \in \mathcal{C}_{/X}$. This yields isomorphisms
    \[
    H^{0}(X; \mathcal{F}) \cong H^{0}(X; \operatorname{Res}(\mathcal{F})),         \check{H}^{0}(\mathcal{U}; \mathcal{F}) \cong \check{H}^{0}(\mathcal{U}; \operatorname{Res}(\mathcal{F})).
    \]
    
    Since $\mathcal{U}$ is admissible with respect to the coverage $\mathcal{S}'$, one has the following chain of isomorphisms
    \[
    H^{0}(X; \mathcal{F}) \cong H^{0}(X; \operatorname{Res}(\mathcal{F})) \cong \check{H}^{0}(\mathcal{U}; \operatorname{Res}(\mathcal{F})) \cong \check{H}^{0}(\mathcal{U};\mathcal{F}).
    \]
\end{example}
\begin{example}\label{appendix_ex_of_adm_cover_Zarisky_site}
    For this example we need a small but interesting subcategory of the category $Top$ of topological spaces and continuous maps. One way to proceed is as follows. Let $\mathcal{C}$ be the full subcategory of $Top$ on all spaces obtained by topologizing subsets of a sufficiently large infinite set. Let $\mathcal{S}$ be the open embeddings. Then an open cover is admissible, as is a locally finite closed cover. The first assertion follows from Example~\ref{appendix_ex_of_adm_cover_site}, and the second one from~\cite[Theorem~II.1.3.1]{Go58}. 
\end{example}
The last example applies in particular to schemes over a given base scheme equipped with the \'etale topology.
\begin{example}\label{appendix_ex_of_adm_cover_etale_site}
    Let $\mathcal{C}$ be a small full subcategory of schemes of finite type over a given base which contains at least one object in every isomorphism class. Take $\mathcal{S}$ to be the \'etale morphisms. Then, by Examples \ref{appendix_ex_of_adm_cover_extenstion_of_sites} and \ref{appendix_ex_of_adm_cover_Zarisky_site}, an open or locally finite closed cover in the Zariski topology is admissible.
\end{example}

Suppose now $\mathcal{S}$ is a coverage on $\mathcal{C}$ and $f:X\to Y$ is a morphism in $\mathcal{C}$. Recall that if $\mathcal{F}\in \mathop{Sh}_Y$, then the \emph{pullback sheaf} $f^{*}\mathcal{F}\in\mathop{Sh}_X$ is the sheafification of the presheaf that takes $(U\to X)$ to $\mathcal{F}(U \to X \xrightarrow{f} Y)$. The functor $\mathcal{F}\mapsto f^*\mathcal{F}$ is the left adjoint of the \emph{pushforward functor} $f_*:\mathop{Sh}_X\to \mathop{Sh}_Y$, that takes $(V\to Y)$ to $\mathcal{F}(X\times_Y V)$.

Consider families of morphisms $\mathcal{U} = \{U_{i} \to X\}_{i \in I}$ and $\mathcal{V} = \{ V_{i} \to Y\}_{i\in I}$ in $\mathcal{C}$, and morphisms $g_{i}: U_{i} \to V_{i},i\in I$ making the following square commute:
\begin{equation}\label{appendix:diagram_of_morphism_in_D}\begin{tikzcd}
	{U_{i}} & {V_{i}} \\
	X & Y.
	\arrow["{g_{i}}", from=1-1, to=1-2]
	\arrow[from=1-1, to=2-1]
	\arrow[from=1-2, to=2-2]
	\arrow["f"', from=2-1, to=2-2]
\end{tikzcd}
\end{equation}

Note that for any presheaf $\mathcal{F}$ on $\mathcal{C}_{/Y}$ there is a map
\begin{equation*}
    \check{H}^{0}(\mathcal{V}; \mathcal{F}) \to \check{H}^{0}(\mathcal{U}; f^{*}\mathcal{F}),
\end{equation*}
which is natural with respect to $\mathcal{F}$. This map is induced by the following composition
\begin{equation*}
    \mathcal{F}(V_{i} \to Y) \to \mathcal{F}(U_{i} \to Y) \to f^{*}\mathcal{F}(U_{i} \to X),
\end{equation*}
where the first morphism is induced by $g_{i}$ from the commutative diagram (\ref{appendix:diagram_of_morphism_in_D}) and the second map is induced by sheafification.

We will now define a new category, which we denote $\mathcal{D}$. Its objects are pairs $(\mathcal{U}, \mathcal{F})$ where $\mathcal{U}$ is an admissible cover of some $X\in \mathcal{C}$, and $\mathcal{F}\in\mathop{Sh}_{X}$. Morphisms exists only between $(\mathcal{U}, f^{*}\mathcal{F})$ and $(\mathcal{V}, \mathcal{F})$ where $f: X\to Y$ is an arbitrary morphism in $\mathcal{C}$, $\mathcal{U} = \{U_{i} \to X\}_{i \in I}$, respectively $\mathcal{V} = \{ V_{i} \to Y\}_{i\in I}$ is an admissible cover of $X$, respectively $Y$, and $\mathcal{F}\in\mathop{Sh}_Y$, in which case the morphisms are given by collections $(g_{i}: U_{i} \to V_{i})_{i\in I}$ such that (\ref{appendix:diagram_of_morphism_in_D}) commutes.

\begin{theorem}\label{appendix:Mayer_Vietoris_main_thrm}
    The assignment $\mathcal{D} \ni (\mathcal{U}, \mathcal{F}) \mapsto (E_{r}^{p,q})\Longrightarrow H^{p+q}(X; \mathcal{F})$ where $E_{2}^{p,q}\cong \check{H}^{p}(\mathcal{U}; \mathcal{F})$ is a functor from $\mathcal{D}$ to spectral sequences of Abelian groups. 
    
    In particular, suppose we have a morphism $f: X\to Y$ in $\mathcal{C}$, admissible covers $\mathcal{U} = \{ U_{1}, U_{2}\},\ \mathcal{V} = \{ V_{1}, V_{2}\}$ of $X$ and $Y$ respectively, and arrows $g_i:U_i\to V_i,i=1,2$ in $\mathcal{C}$ such that (\ref{appendix:diagram_of_morphism_in_D}) commutes. Then this data induces a morphism of the corresponding Mayer-Vietoris long exact sequences.
\end{theorem}
\begin{proof}
    The theorem follows immediately from the discussion above and the functoriality of the Grothendieck spectral sequence $R^{p}F\circ R^{q}G \Longrightarrow R^{p+q}(F\circ G)$ with respect to natural transformations of $G$.
\end{proof}
\end{appendices}

\begin{flushleft}
Alexey Gorinov: Faculty of Mathematics, Higher School of Economics, email: \url{agorinov@hse.ru}, \url{gorinov@mccme.ru}.
\end{flushleft}
\begin{flushleft}
Egor Kosolapov: Higher School of Modern Mathematics, Moscow Institute of Physics and Technology, email: \url{kosolapov.es@phystech.edu}.
\end{flushleft}

\end{document}